\documentclass[leqno]{amsart}
\usepackage{amsthm}
\usepackage{amsfonts}
\usepackage{amsmath}
\usepackage{amssymb}
\usepackage{graphicx}
\usepackage{float} 

\usepackage{mathtools} 

\usepackage{makecell} 

\usepackage{enumerate}
\usepackage[shortlabels]{enumitem}

\makeatletter
\renewcommand\section{\@startsection {section}{1}{\z@}%
                                   {-3.5ex \@plus -1ex \@minus -.2ex}%
                                   {2.3ex \@plus.2ex}%
                                   {\normalfont\large\bfseries\centering}}
\makeatother

\makeatletter
\renewcommand\subsection{\@startsection{subsection}{2}{\z@}%
                                     {-3.25ex\@plus -1ex \@minus -.2ex}%
                                     {1.5ex \@plus .2ex}%
                                     {\normalfont\normalsize\bfseries}}
\makeatother

\makeatletter
\renewcommand{\@secnumfont}{\bfseries} 
\makeatother

\usepackage{longtable,etoolbox}

\makeatletter
\patchcmd{\LT@makecaption}
  {\sbox}
  {\normalsize\sbox}
  {}{}
\makeatother

\allowdisplaybreaks

\usepackage[hidelinks]{hyperref}

\hypersetup{colorlinks,linkcolor={blue},citecolor={red}}

\usepackage{color,soul}
\soulregister\cite7

\usepackage{esint}
\newcommand{\cintpos}{\ointctrclockwise} 

\newcommand{\rd}{\, \mathrm{d}} 

\newcommand{\iverson}[1]{\left[\!\left[ #1 \right]\!\right]}

\newcommand{\NN}{\mathbb{N}} 
\newcommand{\ZZ}{\mathbb{Z}}
\newcommand{\RR}{\mathbb{R}}
\newcommand{\CC}{\mathbb{C}}
\newcommand{\DD}{\mathbb{D}}

\newcommand{\iu}{i\mkern1mu} 

\newcommand{\abs}[1]{\left\lvert #1 \right\rvert}
\newcommand{\card}[1]{\left\lvert #1 \right\rvert}

\newcommand{\maxnorm}[1]{{\left\lVert #1 \right\rVert}_{\max}}

\newtheorem{theorem}{Theorem}[section]
\newtheorem{corollary}{Corollary}[section]
\newtheorem{lemma}{Lemma}[section]
\newtheorem{proposition}{Proposition}[section]

\theoremstyle{definition}
\newtheorem{definition}{Definition}[section]
\newtheorem{remark}{Remark}[section]
\newtheorem{example}{Example}[section]
\newtheorem{open problem}{Open Problem}[section]

\begin{document}

\oddsidemargin 10mm
\evensidemargin 10mm

\thispagestyle{plain}

\vspace{5cc}
\begin{center}

{\Large\bf {Functional Equations for Generalized Collatz Dynamics Using \\ Integral Representations and Residue Calculus} \par 
\rule{0mm}{5mm}\renewcommand{\thefootnote}{}
\footnotetext{ 

\rule{0mm}{0mm}\textbf{Keywords and Phrases.} Functional equations, generating functions, contour integration, residue calculus, several complex variables, Egorychev method, arithmetic dynamics, Collatz $3n+1$ problem.}} 

{\large\it Christos N. Efrem }

\vspace{1cc}
\parbox{26cc}{{\small

\textbf{Abstract.} 
This paper focuses on a wide class of Collatz-type arithmetic dynamics, and presents a systematic derivation of recursive formulas and functional equations satisfied by the associated generating functions. The main tools belong to complex analysis, including contour-integral representations, residue calculus, and Cauchy's integral formulas. The basic approach is inspired by Egorychev's method, which has been used so far to establish combinatorial identities. Moreover, this work generalizes existing results and extends the methodology to multiple dimensions using several complex variables. 

}}
\end{center}


\vspace{1.5cc}

\section{Introduction}

The \emph{classical Collatz map} (or, \emph{$3n+1$ map}) $\psi_{\mathrm{cl}} \colon \NN_0 \to \NN_0$ is defined by
\begin{equation*}
\psi_{\mathrm{cl}}(n) \coloneqq
\begin{cases}
\frac{n}{2} , & \mathrm{if}\ n \equiv 0 \pmod{2}  , \\
3n+1 , & \mathrm{if}\ n \equiv 1 \pmod{2}  . 
\end{cases} 
\end{equation*}
By exploiting the fact that $3n+1$ is even whenever $n$ is odd, we can also define the \emph{shortened Collatz map} (or, \emph{shortened $3n+1$ map}) $\psi \colon \NN_0 \to \NN_0$ as follows
\begin{equation} \label{eq:Shortened_3n+1_map}
\psi(n) \coloneqq 
\begin{cases}
\psi_{\mathrm{cl}}(n) = \frac{n}{2} , & \mathrm{if}\ n \equiv 0 \pmod{2}  , \\
\psi_{\mathrm{cl}}(\psi_{\mathrm{cl}}(n)) = \frac{3n+1}{2} , & \mathrm{if}\ n \equiv 1 \pmod{2}  . 
\end{cases}
\end{equation}
The associated arithmetic dynamics induced by the latter map are given by 
\begin{equation} \label{eq:Shortened_3n+1_map_iterates}
\psi_k(n) \coloneqq \psi(\psi_{k-1}(n))= \psi_{k-1}(\psi(n)) ,  \quad  \forall k \in \mathbb{N} ,
\end{equation}
with initial condition $\psi_0(n) \coloneqq n$. The infamous \emph{Collatz conjecture} \cite{Collatz_1986} asserts that for all integers $n \geq 1$, there exists an integer $k = k(n) \geq 1$ such that $\psi_k(n) = 1$. No proof or disproof has been found so far, despite the significant efforts of many mathematicians and computer scientists.  

An extensive discussion for this challenging problem and its associated results can be found in well-known surveys by Lagarias~\cite{Lagarias_1985,Lagarias_2010}, Wirsching~\cite{Wirsching_2000} and the references therein. Moreover, annotated bibliographies were also given by Lagarias~\cite{Lagarias_2011,Lagarias_2012}. Tao's recent work \cite{Tao_2022} provides a probabilistic (measure-theoretic) point of view by showing that almost all orbits of the Collatz map (in the sense of logarithmic density) attain almost bounded values.

\subsection{Closely Related Work}

Berg and Meinardus \cite{Berg-Meinardus_1994,Berg-Meinardus_1995} studied the generating functions 
$h_k(z) \coloneqq  \sum_{n=0}^{\infty} \psi_k(n) z^n$, $k \in \NN_0$, which are convergent power series for $\abs{z} < 1$. In particular, they showed that these generating functions satisfy the functional recursion 
\begin{equation} \label{eq:Shortened_3n+1_problem_functional_recursion}
h_k(z) = h_{k-1}(z^2) + \frac{1}{3 z^{1/3}} \sum_{l=0}^2 {\zeta^l h_{k-1}(z^{2/3} \zeta^l)} ,  
\end{equation}
for all $k \in \NN$, where $\zeta \coloneqq e^{\iu 2\pi/3}$. According to \cite[Eq.~(20)]{Berg-Meinardus_1995}, equation \eqref{eq:Shortened_3n+1_problem_functional_recursion} can also be written as a contour-integral recursion, i.e., 
\begin{equation} \label{eq:Shortened_3n+1_problem_contour-integral_recursion} 
\begin{split}
h_k(z) & = \frac{1}{2\pi\iu} \cintpos_{\abs{u} = \rho} \frac{(u^2+(u+1)z+z^2)(u-z)}{(u^3-z^2)(u-z^2)} h_{k-1}(u) \rd u   \\
& = \frac{1}{2\pi\iu} \cintpos_{\abs{u} = \rho} \frac{h_{k-1}(u)}{u-z^2} \rd u   + \frac{z}{2\pi\iu} \cintpos_{\abs{u} = \rho} \frac{h_{k-1}(u)}{u^3-z^2} \rd u    ,
\end{split}
\end{equation}  
where the (arbitrary) radius $\rho$ satisfies the inequalities $\abs{z}^{2/3} < \rho < 1$. Moreover, the bivariate generating function $H(z,w) \coloneqq \sum_{k=0}^{\infty} \sum_{n=0}^{\infty}  \psi_k(n) z^n w^k = \sum_{k=0}^{\infty} h_k(z) w^k$, convergent for $\abs{z} < 1$ and $\abs{w} < 2/3$, satisfies the functional equation
\begin{equation} \label{eq:Shortened_3n+1_problem_functional_equation}
H(z,w) = \frac{z}{(1-z)^2} + w \left( H(z^2,w) + \frac{1}{3 z^{1/3}} \sum_{l=0}^2 {\zeta^l H(z^{2/3} \zeta^l,w)}  \right) .
\end{equation}
Note that Eqs.~\eqref{eq:Shortened_3n+1_problem_functional_recursion} and \eqref{eq:Shortened_3n+1_problem_functional_equation} can be obtained from \cite[Eq.~(24)]{Berg-Meinardus_1994} (or, \cite[Eq.~(15)]{Berg-Meinardus_1995}) and \cite[Eq.~(27)]{Berg-Meinardus_1994} (or, \cite[Eq.~(17)]{Berg-Meinardus_1995}), respectively, by applying the variable transformation $z \mapsto z^{1/3}$ therein. Last but not least, Berg and Meinardus proved that the Collatz conjecture is equivalent to the following conjecture \cite[Theorem~5]{Berg-Meinardus_1994}: the only solutions of the homogeneous functional equation 
\begin{equation} \label{eq:Functional_equation_h}
h(z) = h(z^2) + \frac{1}{3 z^{1/3}} \sum_{l=0}^2 {\zeta^l h(z^{2/3} \zeta^l)} 
\end{equation}
that are analytic/holomorphic for $\abs{z} < 1$ are given by $h(z) = c_0 + c_1 \frac{z}{1-z}$, where $c_0, c_1$ are arbitrary complex constants. The functional equation \eqref{eq:Functional_equation_h} can be obtained from~\eqref{eq:Shortened_3n+1_problem_functional_recursion} by replacing $h_k$ and $h_{k-1}$ with $h$; the (unknown) function $h$ can be considered as the limit of $h_k$, as $k \to \infty$, provided that such a limit exists. 

A similar complex-analytic approach was followed by Berg and Opfer~\cite{Berg-Opfer_2013} to investigate the shortened $3n+1$ problem for nonpositive starting values. This problem turns out to be equivalent to the \emph{shortened $3n-1$ problem} with nonnegative starting values, which is generated by the map $\gamma \colon \NN_0 \to \NN_0$ as follows  
\begin{equation} \label{eq:Shortened_3n-1_map}
\gamma(n) \coloneqq 
\begin{cases}
\frac{n}{2} , & \mathrm{if}\ n \equiv 0 \pmod{2}  , \\
\frac{3n-1}{2} , & \mathrm{if}\ n \equiv 1 \pmod{2}  . 
\end{cases}   
\end{equation}
The iterates of this map, $\gamma_k(n)$, are defined analogously to \eqref{eq:Shortened_3n+1_map_iterates}. In this case, the generating functions $\widetilde{h}_k(z) \coloneqq \sum_{n=0}^{\infty} \gamma_k(n) z^n$, $k \in \NN_0$, which are again convergent for $\abs{z} < 1$, satisfy a different functional recursion 
\begin{equation} \label{eq:Shortened_3n-1_problem_functional_recursion}
\widetilde{h}_k(z) = \widetilde{h}_{k-1}(z^2) + \frac{z^{1/3}}{3} \sum_{l=0}^2 {\zeta^{-l} \widetilde{h}_{k-1}(z^{2/3} \zeta^l)}  .
\end{equation} 
In particular, the recursion \eqref{eq:Shortened_3n-1_problem_functional_recursion} results from \cite[Eqs.~(2.3)--(2.4)]{Berg-Opfer_2013} by first replacing $z$ with $z^{1/3}$ in (2.3), and then summing with (2.4); note that the functions $\widetilde{h}_k$ and $\widetilde{h}_{k-1}$ were replaced by the limiting function $h$ therein. 

Furthermore, Siegel~\cite{Siegel_2019,Siegel_PhD_2022,Siegel_2023,Siegel_2024,Siegel_2025} established new functional equations for Collatz-type dynamical systems, which are induced by the so-called ``Hydra map'' (a generalization of the Collatz map). Those equations involve functions that are different from generating functions (power series) on the open unit disk. In particular, his approach relies on complex and $(p,q)$-adic analysis, where $p$ and $q$ are distinct primes, and provides an interesting perspective on the Collatz conjecture.  

Finally, Egorychev's method \cite{Egorychev_1984, Riedel-Mahmoud_2023} is a powerful technique for computing (finite/infinite) combinatorial sums in closed form, thus proving \emph{combinatorial identities}. In particular, the method reduces the considered sum to one/multi-dimensional contour integrals (\emph{integral representation} of the sum) that can be computed by means of \emph{residues}. It is a beautiful and wonderful application of complex analysis (continuous mathematics) in combinatorics (discrete mathematics). The computation algorithm of Egorychev's method \cite[\S 1.1]{Egorychev_1984} includes the following steps: 
\begin{enumerate}[label={\arabic*.}]
	\item Representation of the original expression as a sum of products of combinatorial numbers (e.g., factorials, binomial coefficients and Stirling numbers).
	
	\item Replacement of the combinatorial numbers with their integral representations.
	
	\item Reduction of the products of integrals to multi-dimensional integrals.
	
	\item Interchange of the order of summation and integration (deformation of the integration contour may be necessary to ensure \emph{uniform convergence} of series under the integral sign).
	
	\item Summation of the series in closed form (usually, geometric progressions), thus obtaining an integral representation of the original expression/sum.

	\item Computation of the resulting integral using residue calculus (e.g., iterated integration, variable transformation, and the splitting method with additional/auxiliary variables). 
\end{enumerate}

\subsection{Overview of Main Contributions}

This work provides a new perspective on the analysis of one/multi-dimensional generalized Collatz dynamics. In particular, it develops a unified methodology for proving functional equations and recurrences that are satisfied by their respective generating functions (convergent power series). The fundamental proof techniques are similar to those of Egorychev's method (e.g., contour integrals and residues in one/several complex variables), but now they are used differently in order to establish functional equations/recurrences instead of combinatorial identities. The key results consist of Theorems~\ref{thrm:Theorem_1}~and~\ref{thrm:Theorem_2} for a broad spectrum of arithmetic dynamics in single and multiple dimensions, respectively; see also Definitions~\ref{def:Generalized_Collatz_dynamics}~and~\ref{def:Multi-dimensional_generalized_Collatz_dynamics}. It is emphasized that the purpose of this paper is to generalize the works of Berg, Meinardus and Opfer \cite{Berg-Meinardus_1994,Berg-Meinardus_1995,Berg-Opfer_2013} rather than proving or disproving the Collatz conjecture, which nevertheless remains unsolved.

\subsection{Outline}

The remainder of this paper is organized as follows. Section~\ref{sec:Main_Results} presents the main results with the basic definitions of generalized Collatz dynamics. Afterwards, Section~\ref{sec:Preliminaries_Convergence_of_Multiple_Infinite_Series} deals with convergence issues of multiple infinite series, which play a crucial role in the analysis of power series in several complex variables. Section~\ref{sec:Useful_Lemmas} gives several useful lemmas, and finally Sections~\ref{sec:Proof_Proposition_1}--\ref{sec:Proof_Theorem_2} provide the mathematical proofs of the principal theorems and propositions.

\subsection{Notation}

We write $x \coloneqq y$, or $y \eqqcolon x$, to state that $x$ is by definition equal to $y$ (i.e., $x$ is defined to be equal to $y$). In addition, $\ZZ$, $\RR$ and $\CC$ denote the set of all integers, real and complex numbers, respectively. Calligraphic uppercase letters represent sets. Boldface lowercase (resp., uppercase) letters stand for column-vectors (resp., matrices). Also, $\mathbf{0}$ and $\mathbf{1}$ represent the zero and all-ones vector, respectively (of appropriate dimension). The cardinality of a set is denoted by $\card{\cdot}$, the transpose of a vector/matrix by $[\cdot]^{\top}$, and the degree of a polynomial by $\operatorname{deg}(\cdot)$. The floor function is defined by $\lfloor x \rfloor \coloneqq \max \{k \in \ZZ : k \leq x\}$, for all $x \in \RR$, the absolute value (or, modulus) of the complex number $z$ by $\abs{z}$, the imaginary unit by $\iu \coloneqq \sqrt{-1}$, and the maximum norm of the vector $\mathbf{u} = [u_1,\dots,u_d]^{\top}$ by $\maxnorm{\mathbf{u}} \coloneqq \max \{\abs{u_1},\dots,\abs{u_d}\}$. Some mathematical conventions include: $0^0 \coloneqq 1$, $0! \coloneqq 1$, and $\sum_{j=l}^u \alpha_j \coloneqq 0$, $\prod_{j=l}^u \alpha_j \coloneqq 1$, $\bigcup_{j=l}^u \mathcal{S}_j \coloneqq \varnothing$, $\bigcap_{j=l}^u \mathcal{S}_j \coloneqq \Omega$ (the universal set) whenever $l>u$. The additional notation is summarized in Table~\ref{tbl:List_of_Mathematical_Symbols}, where we assume that $d,d' \in \NN$ (dimensions), $\mathbf{n} = [n_1,\dots,n_d]^{\top} \in \NN_0^d$ (multi-index), $\mathbf{z} = [z_1,\dots,z_d]^{\top} \in \CC^d$ (vector), $f \colon \CC^d \to \CC$ (multivariate function), and $\mathbf{f} \colon \CC^d \to \CC^{d'}$ with $\mathbf{f}(\mathbf{z}) = [f_1(\mathbf{z}),\dots,f_{d'}(\mathbf{z})]^{\top}$ (multivariate vector function).


\renewcommand{\arraystretch}{1.5}
\begin{longtable}{|l|l|}
\caption{List of Mathematical Symbols.}
\label{tbl:List_of_Mathematical_Symbols} \\
\hline 
\multicolumn{1}{|c|}{\textit{Symbol (and its definition)}} & \multicolumn{1}{|c|}{\textit{Description}}  \\
\hline 
$\NN \coloneqq \{1,2,\dots\}$  & the set of all positive integers \\
\hline
$\NN_0 \coloneqq \NN \cup \{0\} = \{0,1,\dots\}$  &  the set of all nonnegative integers \\
\hline
$\RR_{+} \coloneqq \{x \in \RR: x>0\}$  & the set of all positive real numbers \\
\hline
$\DD \coloneqq \{z \in \CC : \abs{z}<1\}$ & the open unit disk in $\CC$  \\   
\hline
$\DD_{*} \coloneqq \DD \setminus\{0\} = \{z \in \CC : 0<\abs{z}<1\}$  &  the punctured open unit disk in $\CC$ \\
\hline
$\prod_{j=1}^d {\mathcal{S}_j} \coloneqq \mathcal{S}_1 \times \cdots \times \mathcal{S}_d$ & \makecell[l]{the Cartesian product of $\mathcal{S}_1,\dots,\mathcal{S}_d$ (if all \\ the sets equal $\mathcal{S}$, then $\mathcal{S}^d \coloneqq \prod_{j=1}^d {\mathcal{S}}$)}  \\
\hline
$\begin{aligned}
\mathcal{D}(\boldsymbol{\rho}) \coloneqq & \, \{\mathbf{z} \in \CC^d : \abs{z_j} < \rho_j, \; j = 1,\dots,d\} \\ 
= &  \, \textstyle\prod_{j=1}^d \mathcal{D}(\rho_j)
\end{aligned}$  &  \makecell[l]{the open poly-disk in $\CC^d$ centered \\ at the origin with poly-radius $\boldsymbol{\rho} \in \RR_{+}^d$ \\ (Cartesian product of $d$ open disks in $\CC$)}   \\
\hline
$\begin{aligned}
\mathcal{C} (\boldsymbol{\rho}) \coloneqq & \, \{\mathbf{z} \in \CC^d : \abs{z_j} = \rho_j, \; j = 1,\dots,d\} \\
= &  \, \textstyle\prod_{j=1}^d \mathcal{C}(\rho_j)
\end{aligned}$  &  \makecell[l]{the poly-circle in $\CC^d$ centered at \\ the origin with poly-radius $\boldsymbol{\rho} \in \RR_{+}^d$ \\ (Cartesian product of $d$ circles in $\CC$)}  \\
\hline
$\begin{aligned}
\overline{\mathcal{D}} (\boldsymbol{\rho}) \coloneqq & \, \{\mathbf{z} \in \CC^d : \abs{z_j} \leq \rho_j, \; j = 1,\dots,d\}  \\ 
= &  \, \mathcal{D} (\boldsymbol{\rho}) \cup \mathcal{C} (\boldsymbol{\rho}) = \textstyle\prod_{j=1}^d \overline{\mathcal{D}}(\rho_j)
\end{aligned}$  & the closure of the open poly-disk $\mathcal{D}(\boldsymbol{\rho})$    \\
\hline
$k \bmod m \coloneqq k - m{\lfloor k/m \rfloor}$   &  \makecell[l]{the remainder of the division of $k \in \ZZ$ by \\ $m \in \NN$, with $k \bmod m \in \{0,\dots,m-1\}$ \\ ($\lfloor k/m \rfloor$ is the quotient of the division)}  \\
\hline
\makecell[l]{$k \equiv \ell \pmod{m} \iff$ \\ $k \bmod m = \ell \bmod m$}  &  \makecell[l]{$k$ is congruent to $\ell$ modulo $m$, \\ where $k,\ell \in \ZZ$ and $m \in \NN$} \\ 
\hline
\makecell[l]{$\mathbf{k} \equiv \boldsymbol{\ell} \pmod{\mathbf{m}} \iff$ \\ $k_j \equiv \ell_j \pmod{m_j},\ \forall j \in \{1,\dots,d\}$}  &  \makecell[l]{component-wise modulo congruence, \\ where $\mathbf{k},\boldsymbol{\ell} \in \ZZ^d$ and $\mathbf{m} \in \NN^d$} \\
\hline
$\begin{aligned}
\mathcal{R}(\mathbf{m}) \coloneqq &  \, \textstyle\prod_{j=1}^d \{0,\dots,m_j-1\}  \\
= &  \, \textstyle\prod_{j=1}^d \mathcal{R}(m_j)
\end{aligned}$  &  \makecell[l]{the set of all (vector) remainders resulted \\ from the component-wise modulo $\mathbf{m} \in \NN^d$}  \\
\hline
$\mathbf{n}! \coloneqq \prod_{j=1}^d ({n_j}!) = n_1! \cdots n_d!$  & the (generalized) factorial of $\mathbf{n}$  \\
\hline
${\lVert \mathbf{n} \rVert} \coloneqq \sum_{j=1}^d \abs{n_j} = \sum_{j=1}^d n_j$ & the $1$-norm of the multi-index $\mathbf{n}$ \\
\hline
$\mathbf{z}_{-j} \coloneqq [z_1,\dots,z_{j-1},z_{j+1},\dots,z_d]^{\top}$  &  \makecell[l]{the vector generated from $\mathbf{z}$ by deleting \\ its $j^{\text{th}}$ component, where $j \in \{1,\dots,d\}$}   \\
\hline
$\mathbf{z}^{\mathbf{n}} \coloneqq \prod_{j=1}^{d} {z_j^{n_j}} = z_1^{n_1} \cdots z_d^{n_d}$ & the monomial produced by $\mathbf{z}$ and $\mathbf{n}$ \\
\hline
$\mathrm{d} \mathbf{z} \coloneqq \prod_{j=1}^d \mathrm{d} z_j = \mathrm{d} z_1 \cdots \mathrm{d} z_d$  &  the (multiple) differential of $\mathbf{z}$  \\
\hline
$f^{(\mathbf{n})}(\mathbf{v}) \coloneqq \left. \frac{\partial^{\lVert \mathbf{n} \rVert} f(\mathbf{z})}{\partial z_1^{n_1} \cdots \partial z_d^{n_d}} \right\rvert_{\mathbf{z} = \mathbf{v}}$   &   \makecell[l]{the $\mathbf{n}^{\text{th}}$-order (complex) partial derivative \\ of $f(\mathbf{z})$ at $\mathbf{z} = \mathbf{v}$, with $f^{(\mathbf{0})}(\mathbf{v}) \coloneqq f(\mathbf{v})$} \\
\hline
$\mathbf{f}^{(\mathbf{n})}(\mathbf{v}) \coloneqq [f_1^{(\mathbf{n})}(\mathbf{v}),\dots,f_{d'}^{(\mathbf{n})}(\mathbf{v})]^{\top}$   &  \makecell[l]{the $\mathbf{n}^{\text{th}}$-order partial derivative of $\mathbf{f}(\mathbf{z})$ \\ at $\mathbf{z} = \mathbf{v}$, with $\mathbf{f}^{(\mathbf{0})}(\mathbf{v}) \coloneqq \mathbf{f}(\mathbf{v})$} \\
\hline
$\cintpos f(\mathbf{z}) \rd \mathbf{z} \coloneqq \cintpos \cdots \cintpos f(\mathbf{z}) \rd z_1 \cdots \rd z_d $  &  \makecell[l]{contour integral of $f$ with positive \\ (counterclockwise) orientation} \\
\hline
$\cintpos \mathbf{f}(\mathbf{z}) \rd \mathbf{z} \coloneqq [\cintpos f_1(\mathbf{z}) \rd \mathbf{z},\dots,\cintpos f_{d'}(\mathbf{z}) \rd \mathbf{z}]^{\top}$  &  component-wise contour integration of $\mathbf{f}$   \\
\hline
$\iverson{P} \coloneqq 
\begin{cases}   
1, & \text{if the statement}\ P\ \text{is true} , \\
0, & \text{if the statement}\ P\ \text{is false} .
\end{cases}$ & the Iverson bracket  \\
\hline
\makecell[l]{$\mathbf{z} = {\mathbf{x} \odot \mathbf{y}} \: \left( = \mathbf{y} \odot \mathbf{x} \right) \iff$ \\ $z_j = x_j y_j, \ \forall j \in \{1,\dots,d\}$}  &  \makecell[l]{the Hadamard/component-wise product \\ of the $d$-dimensional vectors $\mathbf{x}$ and $\mathbf{y}$} \\
\hline 
\end{longtable}

\section{Main Results} \label{sec:Main_Results}

In this section, we present the principal definitions and give the main theorems. Note that the following definitions of generalized Collatz dynamics are similar to (although, slightly different from) those of one/multi-dimensional ``Hydra maps'' used in Siegel's works \cite{Siegel_2019,Siegel_PhD_2022,Siegel_2024}.

\subsection{One-Dimensional Arithmetic Dynamics}

\begin{definition}[Generalized Collatz dynamics] \label{def:Generalized_Collatz_dynamics}
Let $m \in \NN$, $\mathbf{a} = [a_0,\dots,a_{m-1}]^{\top}$, and \linebreak $\mathbf{b} = [b_0,\dots,b_{m-1}]^{\top}$. The \emph{$(m,\mathbf{a},\mathbf{b})$-Collatz map} $t \colon \NN_0 \to \NN_0$ is defined by 
\begin{equation}  \label{eq:Generalized_Collatz_map}
t(n) \coloneqq \begin{cases}
      a_0 n + b_0, & \mathrm{if}\ n \equiv 0 \pmod{m} , \\
      a_1 n + b_1, & \mathrm{if}\ n \equiv 1 \pmod{m} , \\
      \vdots & \vdots \\
      a_{m-1} n + b_{m-1}, & \mathrm{if}\ n \equiv m-1 \pmod{m} , 
\end{cases}      
\end{equation}
whenever the vectors $\mathbf{a}$ and $\mathbf{b}$ satisfy the following conditions:
\begin{enumerate}[start=1,label={(C1.\alph*)},leftmargin=1.15cm] 
	\item \label{cond:1.a} $\lambda_r \coloneqq {a_r m} \in \NN$, and   

	\item \label{cond:1.b} $\mu_r \coloneqq (a_r r + b_r) \in \NN_0$,   
\end{enumerate}
for all $r \in \mathcal{R}(m) \coloneqq \{0,\dots,m-1\}$. The induced \emph{$(m,\mathbf{a},\mathbf{b})$-Collatz dynamics} is defined recursively by the iterates of this map, i.e., 
\begin{equation} \label{eq:Generalized_Collatz_dynamics}
t_k(n) \coloneqq t(t_{k-1}(n))= t_{k-1}(t(n)) ,  \quad  \forall k \in \mathbb{N} ,
\end{equation}
with initial condition $t_0(n) \coloneqq n$.   
\end{definition}

\begin{remark}
By using the Iverson bracket, Eq.~\eqref{eq:Generalized_Collatz_map} can be written as 
\begin{equation*}
t(n) = \sum_{r=0}^{m-1} (a_r n + b_r) \iverson{n \equiv r \pmod{m}} .
\end{equation*}
\end{remark}

\begin{remark}
If $n \in \NN_0$, then $t(n) \in \NN_0$. In particular, given an $n \in \NN_0$ and an $r \in \mathcal{R}(m)$, we have $n \equiv r \pmod{m} \iff n = q m + r$, where $q = \lfloor n/m \rfloor \in \NN_0$. Therefore, if $n \in \NN_0$ and \linebreak $n \equiv r \pmod{m}$, then $t(n) = a_r n + b_r = q(a_r m) + (a_r r + b_r) = {q \lambda_r + \mu_r}  \in \NN_0$, since the conditions \ref{cond:1.a} and \ref{cond:1.b} are satisfied.  
\end{remark}

Next, we can show that the power series $\{\sum_{n=0}^{\infty} t_k(n) z^n\}_{k \in \NN_0}$ and $\sum_{k=0}^{\infty} \sum_{n=0}^{\infty} t_k(n) z^n w^k$ are convergent for sufficiently small $\abs{z}$ and $\abs{w}$, thus having nonzero radii of convergence.    

\begin{proposition} \label{prop:Unconditional_uniform_and_pointwise_convergence_of_GFs}
Let $\rho_z, \rho_w \in \RR_+$ such that $\rho_z < 1$ and $\rho_w < R_w$, where $R_w \coloneqq \min \{ \maxnorm{\mathbf{a}}^{-1},1 \} \linebreak > 0$.\footnote{Note that $\maxnorm{\mathbf{a}} \neq 0$ (equivalently, $\maxnorm{\mathbf{a}} > 0$) due to condition~\ref{cond:1.a}.} Then, the power series $\sum_{n=0}^{\infty} t_k(n) z^n$ is unconditionally uniformly convergent on the disk $\overline{\mathcal{D}}(\rho_z)$ and unconditionally pointwise convergent on the open unit disk $\DD$, for all $k \in \NN_0$. Moreover, the power series $\sum_{k=0}^{\infty} \sum_{n=0}^{\infty}  t_k(n) z^n w^k$ is unconditionally uniformly convergent on the poly-disk $\overline{\mathcal{D}}(\rho_z,\rho_w)$ and unconditionally pointwise convergent on the open poly-disk $\mathcal{D}(1,R_w) = \DD \times {\mathcal{D}(R_w)}$. 
\end{proposition}

\begin{proof}
See Section~\ref{sec:Proof_Proposition_1}. Note that Section~\ref{sec:Preliminaries_Convergence_of_Multiple_Infinite_Series} is a prerequisite, since it contains the necessary definitions and results on the convergence of multiple infinite series. 
\end{proof}

In view of Proposition~\ref{prop:Unconditional_uniform_and_pointwise_convergence_of_GFs}, we can define the generating functions $f_k \colon \DD \to \CC$, for all $k \in \NN_0$, and $F \colon {\DD \times \mathcal{D}(R_w)} \to \CC$ as follows 
\begin{equation} \label{eq:Generating_function_f_k}
f_k(z) \coloneqq \sum_{n=0}^{\infty} t_k(n) z^n , 
\end{equation}
\begin{equation} \label{eq:Generating_function_F}
F(z,w) \coloneqq \sum_{k=0}^{\infty} \sum_{n=0}^{\infty} t_k(n) z^n w^k = \sum_{k=0}^{\infty} f_k(z) w^k  .
\end{equation}
Note that all these generating functions are \emph{continuous} and \emph{holomorphic} on their respective (open) domains, because of Proposition~\ref{prop:Unconditional_uniform_and_pointwise_convergence_of_GFs} and Lemma~\ref{lem:Cauchy_integral_formulas}.

\begin{theorem} \label{thrm:Theorem_1}
Given an $(m,\mathbf{a},\mathbf{b})$-Collatz dynamics, according to Definition~\ref{def:Generalized_Collatz_dynamics}, we have the following statements:
\begin{enumerate}[start=1,label={(S1.\alph*)},leftmargin=1.15cm] 
	\item \label{statement:1.a} Let $z \in \DD_{*} \coloneqq \DD \setminus\{0\}$. Then, there exist unique complex numbers $\{\sigma_{r,l}(z)\}$ and $\{\tau_{r,l}(z)\}$ such that  
\begin{equation} \label{eq:Partial_fraction_decomposition_wrt_u}
\begin{split}
\frac{1}{u^{\mu_r - \lambda_r + 1} (u^{\lambda_r}-z^m)} & = \frac{1}{u^{\mu_r - \lambda_r + 1} \prod_{l=0}^{\lambda_r-1} (u - \phi_{r,l}(z))}   \\
& = \sum_{l=0}^{\mu_r - \lambda_r} {\frac{\sigma_{r,l}(z)}{u^{l+1}}} + \sum_{l=0}^{\lambda_r-1} {\frac{\tau_{r,l}(z)}{u - \phi_{r,l}(z)}}  , 
\end{split}   
\end{equation}
where 
\begin{equation} \label{eq:phi_rj}
\phi_{r,l}(z) \coloneqq z^{m/\lambda_r} e^{\iu 2\pi l / \lambda_r} = z^{a_r^{-1}} \zeta_r^l     
\end{equation}
and $\zeta_r \coloneqq e^{\iu 2\pi / \lambda_r}$, for all $r \in {\mathcal{R}(m)}$ and $l \in \{0,\dots,\lambda_r-1\}$.\footnote{In particular, given any $z \neq 0$ and $r \in {\mathcal{R}(m)}$, the set of all complex solutions of the polynomial equation $u^{\lambda_r} = z^m$ with respect to $u$ is given by $\{\phi_{r,0}(z),\dots,\phi_{r,\lambda_r-1}(z)\}$. Note that $\lambda_r \neq 0$ because of condition \ref{cond:1.a}, thus the $\{\phi_{r,l}(z)\}$ in \eqref{eq:phi_rj} are well defined.} Eq.~\eqref{eq:Partial_fraction_decomposition_wrt_u} is essentially the partial fraction decomposition of its left-hand-side (rational) expression with respect to the complex variable $u$.\footnote{The partial fraction decomposition in \eqref{eq:Partial_fraction_decomposition_wrt_u} is valid even when $\mu_r - \lambda_r < 0 \iff \mu_r - \lambda_r \leq -1$; in this case, the first sum vanishes (i.e., becomes zero) by convention. In addition, the $z$-dependent complex coefficients $\{\sigma_{r,l}(z)\}$ and $\{\tau_{r,l}(z)\}$ can be efficiently computed using computer algebra (symbolic computation).}  
	
	\item \label{statement:1.b} Furthermore, the generating functions $f_k(z)$ defined in \eqref{eq:Generating_function_f_k} satisfy the contour-integral recursion 
\begin{equation} \label{eq:Contour-integral_recursion_f_k}
f_k(z) =  \sum_{r=0}^{m-1} \frac{z^r}{2\pi\iu} \cintpos_{\mathcal{C}(\rho)} \frac{f_{k-1}(u)}{u^{\mu_r - \lambda_r + 1} (u^{\lambda_r}-z^m)}   \rd u   ,
\end{equation}
for any radius $\rho \in \RR_{+}$ such that $\abs{z}^{\maxnorm{\mathbf{a}}^{-1}} < \rho < 1$, as well as the functional recurrence 
\begin{equation} \label{eq:Functional_recurrence_f_k}
f_k(z) = \sum_{r=0}^{m-1} z^r \left( \sum_{l=0}^{\mu_r - \lambda_r} { \frac{\sigma_{r,l}(z)}{l!} f_{k-1}^{(l)}(0) } + \sum_{l=0}^{\lambda_r-1} {{\tau_{r,l}(z)} f_{k-1}(\phi_{r,l}(z))} \right)     ,
\end{equation}
for all $z \in \DD_{*}$ and $k \in \NN$. The initial condition is $f_0(z) = \frac{z}{(1-z)^2}$.\footnote{Eq.~\eqref{eq:Functional_recurrence_f_k} can also be written as 
\begin{equation*}
f_k(z) = \sum_{r=0}^{m-1} z^r \left( \sum_{l=0}^{\mu_r - \lambda_r} {{\sigma_{r,l}(z) \, t_{k-1}(l)}} + \sum_{l=0}^{\lambda_r-1} {{\tau_{r,l}(z)} f_{k-1}(\phi_{r,l}(z))} \right)     , 
\end{equation*}
because \eqref{eq:Coefficients_a_n} yields $t_{k-1}(l) = \frac{f_{k-1}^{(l)}(0)}{l!}$ for all $l \in \NN_0$.} 

	
	\item \label{statement:1.c} Finally, the generating function $F(z,w)$ defined in \eqref{eq:Generating_function_F} satisfies the functional equation 
\begin{equation} \label{eq:Functional_equation_F}
F(z,w) = f_0(z) + w \sum_{r=0}^{m-1} z^r \left( \sum_{l=0}^{\mu_r - \lambda_r} {{\frac{\sigma_{r,l}(z)}{l!} F_{*}^{(l)}(0,w)}} + \sum_{l=0}^{\lambda_r-1} {{\tau_{r,l}(z)} F(\phi_{r,l}(z),w)} \right)     ,
\end{equation}
for all $(z,w) \in {\DD_{*} \times \mathcal{D}(R_w)}$, where $R_w$ is defined in Proposition~\ref{prop:Unconditional_uniform_and_pointwise_convergence_of_GFs} and $F_{*}^{(l)}(0,w) \coloneqq \left. \frac{\partial^{l} F(z,w)}{\partial z^l} \right\rvert_{z=0}$. 

\end{enumerate} 
\end{theorem}

\begin{proof}
See Section~\ref{sec:Proof_Theorem_1}. Note that Sections~\ref{sec:Preliminaries_Convergence_of_Multiple_Infinite_Series}~and~\ref{sec:Useful_Lemmas} are prerequisites.  
\end{proof}

\begin{corollary} \label{cor:Reduction_f_k_and_F}
If $\mu_r - \lambda_r \leq -1$ for all $r \in \mathcal{R}(m)$, then \eqref{eq:Functional_recurrence_f_k}~and~\eqref{eq:Functional_equation_F} reduce to the following equations: 
\begin{equation*}
f_k(z) = \sum_{r=0}^{m-1} z^r \, {\sum_{l=0}^{\lambda_r-1} {{\tau_{r,l}(z)} f_{k-1}(\phi_{r,l}(z))}} ,
\end{equation*}
\begin{equation*}
F(z,w) = f_0(z) + w \sum_{r=0}^{m-1} z^r \, {\sum_{l=0}^{\lambda_r-1} {{\tau_{r,l}(z)} F(\phi_{r,l}(z),w)}} .
\end{equation*}
\end{corollary}

Subsequently, we present applications of Theorem~\ref{thrm:Theorem_1} to well-known arithmetic dynamics, namely, the shortened $3n+1$ and $3n-1$ problems. 

\begin{example}
The \emph{shortened $3n+1$ map} defined in \eqref{eq:Shortened_3n+1_map}: $m=2$, $\mathbf{a}=\left[\frac{1}{2},\frac{3}{2}\right]^{\top}$, $\mathbf{b}=\left[0,\frac{1}{2}\right]^{\top}$ $\implies$ $\boldsymbol{\lambda} = [1,3]^{\top}$, $\boldsymbol{\mu} = [0,2]^{\top}$. Based on \eqref{eq:Partial_fraction_decomposition_wrt_u} and \eqref{eq:phi_rj}, the partial fraction decomposition for $r = 0$, with $\phi_{0,0}(z) = z^2$, is given by 
\begin{equation*}
\frac{1}{u^{\mu_0 - \lambda_0 + 1} (u^{\lambda_0}-z^2)} = \frac{1}{u - z^2} = {\frac{\tau_{0,0}(z)}{u - \phi_{0,0}(z)}}  ,
\end{equation*}
where $\tau_{0,0}(z) = 1$. In a similar way, the partial fraction decomposition for $r = 1$, with $\phi_{1,l}(z) = z^{2/3} \zeta_1^l$, $l \in \{0,1,2\}$, and $\zeta_1 = e^{\iu 2\pi/3}$, is expressed as 
\begin{equation*}
\frac{1}{u^{\mu_1 - \lambda_1 + 1} (u^{\lambda_1}-z^2)} = \frac{1}{u^3 - z^2} = \sum_{l=0}^{2} {\frac{\tau_{1,l}(z)}{u - \phi_{1,l}(z)}}   ,
\end{equation*}
where $\tau_{1,l}(z) = \frac{\zeta_1^l}{3 z^{4/3}}$, for all $l \in \{0,1,2\}$.

Therefore, from \eqref{eq:Contour-integral_recursion_f_k}, we have the contour-integral recursion
\begin{equation*}
f_k(z) =  \frac{1}{2\pi\iu} \cintpos_{\mathcal{C}(\rho)} \frac{f_{k-1}(u)}{u-z^2} \rd u  + \frac{z}{2\pi\iu} \cintpos_{\mathcal{C}(\rho)} \frac{f_{k-1}(u)}{u^3-z^2} \rd u  , 
\end{equation*}
for any radius $\rho \in \RR_{+}$ such that $\abs{z}^{2/3} = \abs{z}^{\maxnorm{\mathbf{a}}^{-1}} < \rho < 1$. From \eqref{eq:Functional_recurrence_f_k} and \eqref{eq:Functional_equation_F}, or Corollary~\ref{cor:Reduction_f_k_and_F}, we also get 
\begin{equation*}
f_k(z) = f_{k-1}(z^2) + \frac{1}{3 z^{1/3}} \sum_{l=0}^2 {\zeta_1^l f_{k-1}(z^{2/3} \zeta_1^l)}  ,
\end{equation*}	
\begin{equation*}
F(z,w) = \frac{z}{(1-z)^2} + w \left( F(z^2,w) + \frac{1}{3 z^{1/3}} \sum_{l=0}^2 {\zeta_1^l F(z^{2/3} \zeta_1^l,w)}  \right) .
\end{equation*}
The last three expressions are in agreement with \eqref{eq:Shortened_3n+1_problem_functional_recursion}, \eqref{eq:Shortened_3n+1_problem_contour-integral_recursion} and \eqref{eq:Shortened_3n+1_problem_functional_equation}, respectively, which have been proved by Berg and Meinardus 
\cite{Berg-Meinardus_1994,Berg-Meinardus_1995}.
\end{example}

\begin{example}
The \emph{shortened $3n-1$ map} defined in \eqref{eq:Shortened_3n-1_map}: $m=2$, $\mathbf{a}=\left[\frac{1}{2},\frac{3}{2}\right]^{\top}$, $\mathbf{b}=\left[0,-\frac{1}{2}\right]^{\top}$ $\implies$ $\boldsymbol{\lambda} = [1,3]^{\top}$, $\boldsymbol{\mu} = [0,1]^{\top}$. Based on \eqref{eq:Partial_fraction_decomposition_wrt_u} and \eqref{eq:phi_rj}, the partial fraction decomposition for $r = 0$ is exactly the same with that of the shortened $3n+1$ problem. However, the partial fraction decomposition for $r = 1$, with $\phi_{1,l}(z) = z^{2/3} \zeta_1^l$, $l \in \{0,1,2\}$, and $\zeta_1 = e^{\iu 2\pi/3}$, is given by 	
\begin{equation*}
\frac{1}{u^{\mu_1 - \lambda_1 + 1} (u^{\lambda_1}-z^2)} = \frac{u}{u^3 - z^2} = \sum_{l=0}^{2} {\frac{\tau_{1,l}(z)}{u - \phi_{1,l}(z)}}   ,
\end{equation*}
where $\tau_{1,l}(z) = \frac{\zeta_1^{2l}}{3 z^{2/3}} = \frac{\zeta_1^{-l}}{3 z^{2/3}}$ (since $\zeta_1^2 = \zeta_1^{-1}$), for all $l \in \{0,1,2\}$.	

Hence, from \eqref{eq:Contour-integral_recursion_f_k}, we have the contour-integral recursion
\begin{equation*}
f_k(z) =  \frac{1}{2\pi\iu} \cintpos_{\mathcal{C}(\rho)} \frac{f_{k-1}(u)}{u-z^2} \rd u  + \frac{z}{2\pi\iu} \cintpos_{\mathcal{C}(\rho)} \frac{u f_{k-1}(u)}{u^3-z^2} \rd u  , 
\end{equation*}
where $\rho \in \RR_{+}$ satisfies the inequalities $\abs{z}^{2/3} = \abs{z}^{\maxnorm{\mathbf{a}}^{-1}} < \rho < 1$. Moreover, from \eqref{eq:Functional_recurrence_f_k} and \eqref{eq:Functional_equation_F}, or Corollary~\ref{cor:Reduction_f_k_and_F}, we obtain  

\begin{equation*}
f_k(z) = f_{k-1}(z^2) + \frac{z^{1/3}}{3} \sum_{l=0}^2 {\zeta_1^{-l} f_{k-1}(z^{2/3} \zeta_1^l)}  ,
\end{equation*}
\begin{equation*}
F(z,w) = \frac{z}{(1-z)^2} + w \left( F(z^2,w) + \frac{z^{1/3}}{3} \sum_{l=0}^2 {\zeta_1^{-l} F(z^{2/3} \zeta_1^l,w)}  \right)  .
\end{equation*}	
The former equation is identical to \eqref{eq:Shortened_3n-1_problem_functional_recursion}, which has been established by Berg and Opfer \cite{Berg-Opfer_2013}. 
\end{example}

\subsection{Multi-Dimensional Arithmetic Dynamics}

In this subsection, we will present an extension of the methodology to higher dimensions, i.e., generalized Collatz dynamics on the lattice $\NN_0^d$ for any $d \in \NN$. In particular, when $d=1$ all the results of the multi-dimensional case reduce to those of the one-dimensional case. 

\begin{definition}[Multi-dimensional generalized Collatz dynamics] \label{def:Multi-dimensional_generalized_Collatz_dynamics}
Let $d \in \NN$, $\mathbf{m} \in \NN^d$, \linebreak $\mathbf{A}_{\circ} = [\mathbf{A}_{\mathbf{r}}]_{\mathbf{r}  \in  \mathcal{R}(\mathbf{m})}$ and $\mathbf{B} = [\mathbf{b}_{\mathbf{r}}]_{\mathbf{r}  \in  \mathcal{R}(\mathbf{m})}$, where the $\mathbf{A}_{\mathbf{r}} = [\mathbf{a}_{\mathbf{r},1},\dots,\mathbf{a}_{\mathbf{r},d}]^{\top}$ are $d \times d$ matrices, with the $\mathbf{a}_{\mathbf{r},j}$ and $\mathbf{b}_{\mathbf{r}}$ being $d$-dimensional vectors. The \emph{$(\mathbf{m},\mathbf{A}_{\circ},\mathbf{B})$-Collatz map} $\mathbf{t} \colon \NN_0^d \to \NN_0^d$ is defined by 
\begin{equation} \label{eq:Multi-dimensional_generalized_Collatz_map}
\mathbf{t}(\mathbf{n}) \coloneqq 
      \mathbf{A}_{\mathbf{r}} \mathbf{n} + \mathbf{b}_{\mathbf{r}}, \quad \mathrm{if}\ \mathbf{r} \in {\mathcal{R}(\mathbf{m})} \ \mathrm{and} \ \mathbf{n} \equiv \mathbf{r} \pmod{\mathbf{m}}     ,
\end{equation}
whenever the matrices $\mathbf{A}_{\circ}$ and $\mathbf{B}$ satisfy the following conditions:
\begin{enumerate}[start=1,label={(C2.\alph*)},leftmargin=1.15cm] 
	\item \label{cond:2.a} $\boldsymbol{\lambda}_{\mathbf{r}} = [\lambda_{\mathbf{r},1},\dots,\lambda_{\mathbf{r},d}]^{\top} \coloneqq {\mathbf{A}_{\mathbf{r}} \mathbf{m}} \in \NN^d$, and   

	\item \label{cond:2.b} $\boldsymbol{\mu}_{\mathbf{r}} = [\mu_{\mathbf{r},1},\dots,\mu_{\mathbf{r},d}]^{\top} \coloneqq {\mathbf{A}_{\mathbf{r}} \mathbf{r} + \mathbf{b}_{\mathbf{r}}} \in \NN_0^d$,   
\end{enumerate}
for all $\mathbf{r} \in {\mathcal{R}(\mathbf{m})}$. The induced \emph{$(\mathbf{m},\mathbf{A}_{\circ},\mathbf{B})$-Collatz dynamics} is defined recursively by the iterates of this map, i.e., 
\begin{equation} \label{eq:Multi-dimensional_generalized_Collatz_dynamics}
\mathbf{t}_k(\mathbf{n}) \coloneqq \mathbf{t}(\mathbf{t}_{k-1}(\mathbf{n})) = \mathbf{t}_{k-1}(\mathbf{t}(\mathbf{n})) ,  \quad  \forall k \in \mathbb{N} ,
\end{equation}
with initial condition $\mathbf{t}_0(\mathbf{n}) \coloneqq \mathbf{n}$. 
\end{definition}

\begin{remark}
By leveraging the Iverson bracket, Eq.~\eqref{eq:Multi-dimensional_generalized_Collatz_map} can be expressed as 
\begin{equation*}
\mathbf{t}(\mathbf{n}) =  \sum_{\mathbf{r}  \in  {\mathcal{R}(\mathbf{m})}}  (\mathbf{A}_{\mathbf{r}} \mathbf{n} + \mathbf{b}_{\mathbf{r}}) \iverson{\mathbf{n} \equiv \mathbf{r} \pmod{\mathbf{m}}}  .
\end{equation*}
\end{remark}

\begin{remark}
If $\mathbf{n} \in \NN_0^d$, then $\mathbf{t}(\mathbf{n}) \in \NN_0^d$. Specifically, given an $\mathbf{n} \in \NN_0^d$ and a vector $\mathbf{r} \in {\mathcal{R}(\mathbf{m})}$, we have $\mathbf{n} \equiv \mathbf{r} \pmod{\mathbf{m}} \iff \mathbf{n} = {\mathbf{q} \odot \mathbf{m} + \mathbf{r}} \: \left( = \mathbf{m} \odot \mathbf{q} + \mathbf{r} \right)$, where $\mathbf{q} = [{\lfloor n_1/m_1 \rfloor},\dots,{\lfloor n_d/m_d \rfloor}]^{\top} \in \NN_0^d$. As a result, if $\mathbf{n} \in \NN_0^d$ and $\mathbf{n} \equiv \mathbf{r} \pmod{\mathbf{m}}$, then $\mathbf{t}(\mathbf{n}) = \mathbf{A}_{\mathbf{r}} \mathbf{n} + \mathbf{b}_{\mathbf{r}} = (\mathbf{A}_{\mathbf{r}} \mathbf{m}) \odot \mathbf{q} + (\mathbf{A}_{\mathbf{r}} \mathbf{r} + \mathbf{b}_{\mathbf{r}}) = \mathbf{q} \odot \boldsymbol{\lambda}_{\mathbf{r}} + \boldsymbol{\mu}_{\mathbf{r}}   \in \NN_0^d$, because of the conditions \ref{cond:2.a} and \ref{cond:2.b}.   
\end{remark}

Afterwards, we can prove that the vector power series $\{\sum_{\mathbf{n} \in \NN_0^d} {{\mathbf{t}_k(\mathbf{n})} \, \mathbf{z}^{\mathbf{n}}}\}_{k \in \NN_0}$ and \linebreak $\sum_{k \in \NN_0} \sum_{\mathbf{n} \in \NN_0^d} {\mathbf{t}_k(\mathbf{n}) \, \mathbf{z}^{\mathbf{n}} w^k}$ are convergent for sufficiently small $\maxnorm{\mathbf{z}}$ and $\abs{w}$, thus having nonzero radii of convergence.   

\begin{proposition} \label{prop:Unconditional_uniform_and_pointwise_convergence_of_vector_GFs}
Let $\rho_{\mathbf{z}}, \rho_w \in \RR_+$ such that $\rho_{\mathbf{z}} < 1$ and $\rho_w < R_w$, where $R_w \coloneqq \linebreak \min \{ (d \maxnorm{\mathbf{A}_{\circ}})^{-1},1 \} > 0$.\footnote{Note that $\maxnorm{\mathbf{A}_{\circ}} \neq 0$ (equivalently, $\maxnorm{\mathbf{A}_{\circ}} > 0$) because of condition~\ref{cond:2.a}.} Then, the vector power series $\sum_{\mathbf{n} \in \NN_0^d} {{\mathbf{t}_k(\mathbf{n})} \, \mathbf{z}^{\mathbf{n}}}$ is unconditionally uniformly convergent on the poly-disk $\overline{\mathcal{D}}(\rho_{\mathbf{z}} \mathbf{1}) = \{\mathbf{z} \in \CC^d : \maxnorm{\mathbf{z}} \leq \rho_{\mathbf{z}}\}$ and unconditionally pointwise convergent on the open unit poly-disk $\DD^d$, for all $k \in \NN_0$. Furthermore, the vector power series $\sum_{k \in \NN_0} \sum_{\mathbf{n} \in \NN_0^d} {\mathbf{t}_k(\mathbf{n}) \, \mathbf{z}^{\mathbf{n}} w^k}$ is unconditionally uniformly convergent on the poly-disk $\overline{\mathcal{D}}(\rho_{\mathbf{z}} \mathbf{1},\rho_w) = \{(\mathbf{z},w) \in \CC^{d+1} : \maxnorm{\mathbf{z}} \leq \rho_{\mathbf{z}}, \, \abs{w} \leq \rho_w\}$ and unconditionally pointwise convergent on the open poly-disk $\mathcal{D}(\mathbf{1},R_w) = {\DD^d} \times {\mathcal{D}(R_w)}$.\footnote{When we say that a property holds for a vector power series, we mean that the property is true for all its component power series.}  
\end{proposition}

\begin{proof}
See Section~\ref{sec:Proof_Proposition_2}. Note that Section~\ref{sec:Preliminaries_Convergence_of_Multiple_Infinite_Series} is a prerequisite, since it contains the necessary definitions and results on the convergence of multiple infinite series.
\end{proof}

In view of Proposition~\ref{prop:Unconditional_uniform_and_pointwise_convergence_of_vector_GFs}, we can define the vector generating functions of several complex variables $\mathbf{f}_k \colon \DD^d \to \CC^d$, for all $k \in \NN_0$, and $\mathbf{F} \colon {\DD^d \times \mathcal{D}(R_w)} \to \CC^d$ as follows 
\begin{equation} \label{eq:Vector_generating_function_f_k}
\mathbf{f}_k(\mathbf{z}) \coloneqq \sum_{\mathbf{n} \in \NN_0^d} {{\mathbf{t}_k(\mathbf{n})} \, \mathbf{z}^{\mathbf{n}}}    ,  
\end{equation} 
\begin{equation} \label{eq:Vector_generating_function_F}
\mathbf{F}(\mathbf{z},w) \coloneqq \sum_{k \in \NN_0} \sum_{\mathbf{n} \in \NN_0^d} {\mathbf{t}_k(\mathbf{n}) \, \mathbf{z}^{\mathbf{n}} w^k} = \sum_{k \in \NN_0} {\mathbf{f}_k(\mathbf{z}) \, w^k}    .
\end{equation}
In particular, $\mathbf{t}_k(\mathbf{n}) = [t_{1,k}(\mathbf{n}),\dots,t_{d,k}(\mathbf{n})]^{\top}$, $\mathbf{f}_k(\mathbf{z}) = [f_{1,k}(\mathbf{z}),\dots,f_{d,k}(\mathbf{z})]^{\top}$ and \linebreak $\mathbf{F}(\mathbf{z},w) = [F_1(\mathbf{z},w),\dots,F_d(\mathbf{z},w)]^{\top}$, where $f_{j,k}(\mathbf{z}) \coloneqq \sum_{\mathbf{n} \in \NN_0^d} {{t_{j,k}(\mathbf{n})} \, \mathbf{z}^{\mathbf{n}}}$ and \linebreak $F_j(\mathbf{z},w) \coloneqq \sum_{k \in \NN_0} \sum_{\mathbf{n} \in \NN_0^d} {t_{j,k}(\mathbf{n}) \, \mathbf{z}^{\mathbf{n}} w^k} = \sum_{k \in \NN_0} {f_{j,k}(\mathbf{z}) \, w^k}$, for all $j \in \{1,\dots,d\}$. Note that all these vector generating functions are \emph{continuous} and \emph{holomorphic} on their respective (open) domains, because of Proposition~\ref{prop:Unconditional_uniform_and_pointwise_convergence_of_vector_GFs} and Lemma~\ref{lem:Cauchy_integral_formulas}.

\begin{theorem} \label{thrm:Theorem_2}
Given an $(\mathbf{m},\mathbf{A}_{\circ},\mathbf{B})$-Collatz dynamics, according to Definition~\ref{def:Multi-dimensional_generalized_Collatz_dynamics}, we have the following statements:
\begin{enumerate}[start=1,label={(S2.\alph*)},leftmargin=1.15cm] 
	\item \label{statement:2.a} Let $\textbf{z} \in \DD_{*}^d \coloneqq (\DD \setminus\{0\})^d$. Then, there exist unique complex numbers $\{\sigma_{\mathbf{r},j,l}(z_j)\}$ and $\{\tau_{\mathbf{r},j,l}(z_j)\}$, a finite nonempty set $\mathcal{L}(\mathbf{r}) \subseteq \NN_0^{d+1}$, with $(\mathbf{0},0) \in \mathcal{L}(\mathbf{r})$, complex numbers $\{\eta_{\mathbf{r},\boldsymbol{\ell},\nu}(\mathbf{z})\}$, and $d$-dimensional complex vectors $\{\boldsymbol{\varphi}_{\mathbf{r},\boldsymbol{\ell},\nu}(\mathbf{z})\}$ such that 
\begin{equation} \label{eq:Partial_fraction_decomposition_wrt_u_multi-dimensional}
\begin{split}
\prod_{j=1}^d {\frac{1}{u_j^{\mu_{\mathbf{r},j} - \lambda_{\mathbf{r},j} + 1} (u_j^{\lambda_{\mathbf{r},j}}-z_j^{m_j})}} & = \prod_{j=1}^d {\frac{1}{u_j^{\mu_{\mathbf{r},j} - \lambda_{\mathbf{r},j} + 1} \prod_{l=0}^{\lambda_{\mathbf{r},j}-1} (u_j - \phi_{\mathbf{r},j,l}(z_j))}}    \\
& = \prod_{j=1}^d {\left(  \sum_{l=0}^{\mu_{\mathbf{r},j} - \lambda_{\mathbf{r},j}} {\frac{\sigma_{\mathbf{r},j,l}(z_j)}{u_j^{l+1}}} + \sum_{l=0}^{\lambda_{\mathbf{r},j}-1} {\frac{\tau_{\mathbf{r},j,l}(z_j)}{u_j - \phi_{\mathbf{r},j,l}(z_j)}}  \right)}   \\
& = \sum_{(\boldsymbol{\ell},\nu) \in \mathcal{L}(\mathbf{r})} {\frac{\eta_{\mathbf{r},\boldsymbol{\ell},\nu}(\mathbf{z})}{\left( \mathbf{u} - \boldsymbol{\varphi}_{\mathbf{r},\boldsymbol{\ell},\nu}(\mathbf{z}) \right)^{\boldsymbol{\ell}+\mathbf{1}}}}  , 
\end{split}   
\end{equation}
where 
\begin{equation} \label{eq:phi_rjl}
\phi_{\mathbf{r},j,l}(z_j) \coloneqq z_j^{m_j/\lambda_{\mathbf{r},j}} e^{\iu 2\pi l / \lambda_{\mathbf{r},j}} = z_j^{m_j/\lambda_{\mathbf{r},j}} \zeta_{\mathbf{r},j}^l   
\end{equation}
and $\zeta_{\mathbf{r},j} \coloneqq e^{\iu 2\pi / \lambda_{\mathbf{r},j}}$, for all $\mathbf{r} \in {\mathcal{R}(\mathbf{m})}$, $j \in \{1,\dots,d\}$, $l \in \{0,\dots,\lambda_{\mathbf{r},j}-1\}$. Note that the first and second equality in \eqref{eq:Partial_fraction_decomposition_wrt_u_multi-dimensional} follow from \eqref{eq:Partial_fraction_decomposition_wrt_u}, and the $\{\phi_{\mathbf{r},j,l}(z_j)\}$ in \eqref{eq:phi_rjl} are well defined because $\lambda_{\mathbf{r},j} \coloneqq \mathbf{a}_{\mathbf{r},j}^{\top} \mathbf{m} \neq 0$ based on condition \ref{cond:2.a}.\footnote{Given $\{\sigma_{\mathbf{r},j,l}(z_j)\}$, $\{\tau_{\mathbf{r},j,l}(z_j)\}$ and $\{\phi_{\mathbf{r},j,l}(z_j)\}$, we can compute recursively the $\mathcal{L}(\mathbf{r})$, $\{\eta_{\mathbf{r},\boldsymbol{\ell},\nu}(\mathbf{z})\}$ and $\{\boldsymbol{\varphi}_{\mathbf{r},\boldsymbol{\ell},\nu}(\mathbf{z})\}$ based on the inductive proof given in Section~\ref{sec:Proof_Statement_2.a}.}

	\item \label{statement:2.b} Moreover, the vector generating functions $\mathbf{f}_k(\mathbf{z})$ defined in \eqref{eq:Vector_generating_function_f_k} satisfy the contour-integral recursion 
\begin{equation} \label{eq:Contour-integral_recursion_vector_f_k}
\mathbf{f}_k(\mathbf{z}) = \sum_{\mathbf{r} \in {\mathcal{R}(\mathbf{m})}} \frac{\mathbf{z}^\mathbf{r}}{(2\pi\iu)^d} \cintpos_{\mathcal{C}(\boldsymbol{\rho})}  \left( \prod_{j=1}^d {\frac{1}{u_j^{\mu_{\mathbf{r},j} - \lambda_{\mathbf{r},j} + 1} (u_j^{\lambda_{\mathbf{r},j}}-z_j^{m_j})}} \right)  \mathbf{f}_{k-1}(\mathbf{u})  \rd \mathbf{u}    ,
\end{equation} 
for any poly-radius $\boldsymbol{\rho} \in \RR_{+}^d$ such that $\abs{z_j}^{m_j/\lambda_j^{\max}} < \rho_j < 1$, where $\lambda_j^{\max} \coloneqq \linebreak \max\{ \lambda_{\mathbf{r},j} : \mathbf{r} \in \mathcal{R}(\mathbf{m}) \}$, for all $j \in \{1,\dots,d\}$, as well as the functional recurrence 
\begin{equation} \label{eq:Functional_recurrence_vector_f_k}
\mathbf{f}_k(\mathbf{z}) = \sum_{\mathbf{r} \in {\mathcal{R}(\mathbf{m})}} \mathbf{z}^\mathbf{r} \sum_{(\boldsymbol{\ell},\nu) \in \mathcal{L}(\mathbf{r})}  \frac{\eta_{\mathbf{r},\boldsymbol{\ell},\nu}(\mathbf{z})}{\boldsymbol{\ell}!}  {\mathbf{f}_{k-1}^{(\boldsymbol{\ell})}(\boldsymbol{\varphi}_{\mathbf{r},\boldsymbol{\ell},\nu}(\mathbf{z}))}     ,
\end{equation}
for all $\textbf{z} \in \DD_{*}^d$ and $k \in \NN$. The initial condition is $\mathbf{f}_0(\mathbf{z}) = [f_{1,0}(\mathbf{z}),\dots,f_{d,0}(\mathbf{z})]^{\top}$, where $f_{j,0}(\textbf{z}) = \frac{z_j}{(1-z_j)^2 \prod_{l \neq j} {(1-z_l)}}$ for all $j \in \{1,\dots,d\}$. 

	
	\item \label{statement:2.c} Finally, the vector generating function $\mathbf{F}(\mathbf{z},w)$ defined in \eqref{eq:Vector_generating_function_F} satisfies the functional equation 
\begin{equation} \label{eq:Functional_equation_vector_F}
\mathbf{F}(\mathbf{z},w) = \mathbf{f}_0(\mathbf{z}) + w \sum_{\mathbf{r} \in {\mathcal{R}(\mathbf{m})}} \mathbf{z}^\mathbf{r} \sum_{(\boldsymbol{\ell},\nu) \in \mathcal{L}(\mathbf{r})}  \frac{\eta_{\mathbf{r},\boldsymbol{\ell},\nu}(\mathbf{z})}{\boldsymbol{\ell}!} \mathbf{F}_{*}^{(\boldsymbol{\ell})}(\boldsymbol{\varphi}_{\mathbf{r},\boldsymbol{\ell},\nu}(\mathbf{z}),w)     ,
\end{equation}
for all $(\mathbf{z},w) \in {\DD_{*}^d \times \mathcal{D}(R_w)}$, where $R_w$ is defined in Proposition~\ref{prop:Unconditional_uniform_and_pointwise_convergence_of_vector_GFs} and $\mathbf{F}_{*}^{(\boldsymbol{\ell})}(\mathbf{v},w) \coloneqq \left. \frac{\partial^{\lVert \boldsymbol{\ell} \rVert} \mathbf{F}(\mathbf{z},w)}{\partial z_1^{\ell_1} \cdots \partial z_d^{\ell_d}} \right\rvert_{\mathbf{z}=\mathbf{v}}$.  
\end{enumerate}
\end{theorem}

\begin{proof}
See Section~\ref{sec:Proof_Theorem_2}. Note that Sections~\ref{sec:Preliminaries_Convergence_of_Multiple_Infinite_Series}~and~\ref{sec:Useful_Lemmas} are prerequisites.
\end{proof}

\begin{corollary}
If $\mu_{\mathbf{r},j} \leq \lambda_{\mathbf{r},j}$ (including the case with $\mu_{\mathbf{r},j} = \lambda_{\mathbf{r},j}$) for all $\mathbf{r} \in {\mathcal{R}(\mathbf{m})}$ and $j \in \{1,\dots,d\}$, then \eqref{eq:Functional_recurrence_vector_f_k}~and~\eqref{eq:Functional_equation_vector_F} reduce to the following equations: 
\begin{equation*}
\mathbf{f}_k(\mathbf{z}) = \sum_{\mathbf{r} \in {\mathcal{R}(\mathbf{m})}} \mathbf{z}^\mathbf{r} \sum_{\nu \in \mathcal{N}(\mathbf{r})}  {\widetilde{\eta}_{\mathbf{r},\nu}(\mathbf{z})}  {\mathbf{f}_{k-1}(\widetilde{\boldsymbol{\varphi}}_{\mathbf{r},\nu}(\mathbf{z}))}    ,
\end{equation*}
\begin{equation*}
\mathbf{F}(\mathbf{z},w) =  \mathbf{f}_0(\mathbf{z}) + w \sum_{\mathbf{r} \in {\mathcal{R}(\mathbf{m})}} \mathbf{z}^\mathbf{r} \sum_{\nu \in \mathcal{N}(\mathbf{r})}  {\widetilde{\eta}_{\mathbf{r},\nu}(\mathbf{z})} \mathbf{F}(\widetilde{\boldsymbol{\varphi}}_{\mathbf{r},\nu}(\mathbf{z}),w)   ,
\end{equation*}
where $\mathcal{N}(\mathbf{r}) \coloneqq \{ \nu \in \NN_0 : (\mathbf{0},\nu) \in \mathcal{L}(\mathbf{r}) \}$, $\widetilde{\eta}_{\mathbf{r},\nu}(\mathbf{z}) \coloneqq \eta_{\mathbf{r},\mathbf{0},\nu}(\mathbf{z})$, $\widetilde{\boldsymbol{\varphi}}_{\mathbf{r},\nu}(\mathbf{z}) \coloneqq \boldsymbol{\varphi}_{\mathbf{r},\mathbf{0},\nu}(\mathbf{z})$.
\end{corollary}

\begin{proof}
Immediate consequence of Theorem~\ref{thrm:Theorem_2}, since in this case the multi-dimensional partial fraction decomposition in \eqref{eq:Partial_fraction_decomposition_wrt_u_multi-dimensional} becomes  
\begin{equation*} 
\begin{split}
\prod_{j=1}^d {\frac{1}{u_j^{\mu_{\mathbf{r},j} - \lambda_{\mathbf{r},j} + 1} (u_j^{\lambda_{\mathbf{r},j}}-z_j^{m_j})}} & = \sum_{(\mathbf{0},\nu) \in \mathcal{L}(\mathbf{r})} {\frac{\eta_{\mathbf{r},\mathbf{0},\nu}(\mathbf{z})}{\left( \mathbf{u} - \boldsymbol{\varphi}_{\mathbf{r},\mathbf{0},\nu}(\mathbf{z}) \right)^{\mathbf{1}}}}   \\
& = \sum_{\nu \in \mathcal{N}(\mathbf{r})} {\frac{\widetilde{\eta}_{\mathbf{r},\nu}(\mathbf{z})}{\left( \mathbf{u} - \widetilde{\boldsymbol{\varphi}}_{\mathbf{r},\nu}(\mathbf{z}) \right)^{\mathbf{1}}}}   .  
\end{split}
\end{equation*}
\end{proof}

\section{Preliminaries: Convergence of Multiple Infinite Series} \label{sec:Preliminaries_Convergence_of_Multiple_Infinite_Series}

In contrast to single infinite series, when dealing with multiple infinite series there is no standard way to define its (finite) partial sums based on the natural ordering of its terms. As a result, appropriate definitions of convergence are needed.

\subsection{Multiple Series of Complex Numbers}

\begin{definition}[Absolute convergence of multiple infinite series {\cite[p.\,6]{Laurent-Thiebaut_2011}}, {\cite[p.\,13]{Range_1986}}]
Let $d \in \NN$ and $(\alpha_{\mathbf{n}})_{\mathbf{n} \in \NN_0^d}$ be a sequence of complex numbers. The infinite series $\sum_{\mathbf{n} \in \NN_0^d} {\alpha_{\mathbf{n}}}$ is called \emph{absolutely convergent} if, and only if,
\begin{equation*}
\sum_{\mathbf{n} \in \NN_0^d} \abs{\alpha_{\mathbf{n}}} \coloneqq  \sup \left\{ \sum_{\mathbf{n} \in \mathcal{F}} \abs{\alpha_{\mathbf{n}}} : \mathcal{F} \subseteq \NN_0^d,\, \mathcal{F}\ \mathrm{is\ finite} \right\} < \infty  .
\end{equation*} 
\end{definition}

Note that this definition is a generalization of the (standard) absolute convergence for single ($d=1$) infinite series: $\sum_{n \in \NN_0} \abs{\alpha_n} < \infty$ or, equivalently, the series $\sum_{n \in \NN_0} \abs{\alpha_n}$ converges. 

\begin{definition}[Unconditional convergence of multiple infinite series]
An infinite series of complex numbers $\sum_{\mathbf{n} \in \NN_0^d} {\alpha_{\mathbf{n}}}$ is said to be \emph{unconditionally convergent} if, and only if, there exists an $\ell \in \CC$ such that, for any bijection $\boldsymbol{\sigma} \colon \NN_0 \to \NN_0^d$, the ordinary series $\sum_{j=0}^{\infty} {\alpha_{\boldsymbol{\sigma}(j)}}$ converges, in the usual sense, to $\ell$. The unique complex number $\ell$, which is independent of (or, invariant under) any rearrangement of the series' terms, is called the limit/sum of the series and we write $\sum_{\mathbf{n} \in \NN_0^d} {\alpha_{\mathbf{n}}} = \ell$. 
\end{definition}

The simplest bijection $\boldsymbol{\sigma}$ is the lexicographic order on $\NN_0^d$. In particular, we have $n_1 < n_2$ if and only if $\boldsymbol{\sigma}(n_1) {\prec}_{\mathrm{lex}} \boldsymbol{\sigma}(n_2)$, i.e., $\boldsymbol{\sigma}(n_1)$ precedes $\boldsymbol{\sigma}(n_2)$ in the lexicographic (total) order ${\prec}_{\mathrm{lex}}$. In general, every bijection $\boldsymbol{\sigma}'$ is a permutation of the lexicographic order on $\NN_0^d$.

\begin{proposition}[Absolute convergence implies unconditional convergence] \label{prop:Absolute_convergence_implies_unconditional_convergence}
If a multiple series of complex numbers $\sum_{\mathbf{n} \in \NN_0^d} {\alpha_{\mathbf{n}}}$ is absolutely convergent, then it is unconditionally convergent. 
\end{proposition}

\begin{proof}
For any bijection $\boldsymbol{\sigma} \colon \NN_0 \to \NN_0^d$, we have $\sum_{j=0}^{\infty} \abs{\alpha_{\boldsymbol{\sigma}(j)}} \leq \sum_{\mathbf{n} \in \NN_0^d} \abs{\alpha_{\mathbf{n}}} < \infty$. Since absolute convergence implies convergence (in the standard sense), the series $\sum_{j=0}^{\infty} {\alpha_{\boldsymbol{\sigma}(j)}}$ converges to some ${\ell}_{\boldsymbol{\sigma}} \in \CC$, i.e., $\sum_{j=0}^{\infty} {\alpha_{\boldsymbol{\sigma}(j)}} = {\ell}_{\boldsymbol{\sigma}}$. Now, it remains to show that the limit ${\ell}_{\boldsymbol{\sigma}}$ is independent of $\boldsymbol{\sigma}$. For this purpose, let $\boldsymbol{\sigma}$ and $\boldsymbol{\sigma}'$ be two such bijections (arbitrarily chosen) with $\sum_{j=0}^{\infty} {\alpha_{\boldsymbol{\sigma}(j)}}  =  {\ell}_{\boldsymbol{\sigma}}$ and $\sum_{j=0}^{\infty} {\alpha_{\boldsymbol{\sigma}'(j)}} = {\ell}_{\boldsymbol{\sigma}'}$; we will prove that ${\ell}_{\boldsymbol{\sigma}} = {\ell}_{\boldsymbol{\sigma}'}$. Let $\epsilon > 0$ be fixed, but arbitrary. The series $\sum_{j=0}^{\infty} \abs{\alpha_{\boldsymbol{\sigma}(j)}}$ converges and so, by the Cauchy convergence criterion \cite[Theorem~3.22]{Rudin_1976}, there is an $m_0 = m_0(\epsilon) \in \NN_0$ such that $\abs{\sum_{j = m_0+1}^m \abs{\alpha_{\boldsymbol{\sigma}(j)}}} = \sum_{j = m_0+1}^m \abs{\alpha_{\boldsymbol{\sigma}(j)}} \leq \epsilon$ for every integer $m \geq m_0+1$ (this is also true when $m < m_0+1$, because the latter sum equals $0 \leq \epsilon$). Let $n_0 = n_0(m_0) = n_0(\epsilon) \in \NN_0$ be sufficiently large so that $\{{\boldsymbol{\sigma}(0)},\dots,{\boldsymbol{\sigma}(m_0)}\} \subseteq \{{\boldsymbol{\sigma}'(0)},\dots,{\boldsymbol{\sigma}'(n_0)}\}$; such an $n_0$ always exists since $\boldsymbol{\sigma}$ and $\boldsymbol{\sigma}'$ are bijections. Then, given any integer $n \geq n_0$ and any sufficiently large $m \in \NN_0$ such that $\{{\boldsymbol{\sigma}'(0)},\dots,{\boldsymbol{\sigma}'(n)}\} \subseteq \{{\boldsymbol{\sigma}(0)},\dots,{\boldsymbol{\sigma}(m)}\}$, we have 
\begin{equation*}
\begin{split}
\abs{\sum_{j=0}^n {\alpha_{\boldsymbol{\sigma}'(j)}} - \sum_{j=0}^m {\alpha_{\boldsymbol{\sigma}(j)}}} & = \abs{\sum_{j=0}^m {\alpha_{\boldsymbol{\sigma}(j)}} - \sum_{j=0}^n {\alpha_{\boldsymbol{\sigma}'(j)}}} = \abs{\sum_{j=0}^m {\alpha_{\boldsymbol{\sigma}(j)}} {x_{j,n}}}  \\
& = \abs{\sum_{j = m_0+1}^m {\alpha_{\boldsymbol{\sigma}(j)}} {x_{j,n}}}  \leq \sum_{j = m_0+1}^m {\abs{\alpha_{\boldsymbol{\sigma}(j)} {x_{j,n}}}}  \\
& = \sum_{j = m_0+1}^m {\abs{\alpha_{\boldsymbol{\sigma}(j)}} {x_{j,n}}}  \leq  \sum_{j = m_0+1}^m \abs{\alpha_{\boldsymbol{\sigma}(j)}} \leq \epsilon ,
\end{split}
\end{equation*}
where $(x_{j,n})_{j=0}^m$ is a binary sequence, i.e., $x_{j,n} \in \{0,1\}$, such that $x_{j,n} = 0$ if and only if ${\boldsymbol{\sigma}(j)} \in \{{\boldsymbol{\sigma}'(0)},\dots,{\boldsymbol{\sigma}'(n)}\}$. For every $j \in \{0,\dots,m_0\}$, we have $x_{j,n}=0$ since ${\boldsymbol{\sigma}(j)} \in \{{\boldsymbol{\sigma}'(0)},\dots,{\boldsymbol{\sigma}'(n_0)}\} \subseteq \{{\boldsymbol{\sigma}'(0)},\dots,{\boldsymbol{\sigma}'(n)}\}$ due to the selection of $n_0$. Note that the first inequality follows from the triangle inequality. Given the integer $n$ and by taking the limit as $m \to \infty$, we obtain 
\begin{equation*}
\abs{\sum_{j=0}^n {\alpha_{\boldsymbol{\sigma}'(j)}} - {\ell}_{\boldsymbol{\sigma}}} \leq \epsilon   ,
\end{equation*}
because $\abs{x}$ is a continuous function and $\lim_{m \to \infty} {\sum_{j=0}^m {\alpha_{\boldsymbol{\sigma}(j)}}} = {\ell}_{\boldsymbol{\sigma}}$. Consequently, we have shown that for any $\epsilon > 0$, there is an $n_0 = n_0(\epsilon) \in \NN_0$ such that: if $n_0 \leq n \in \NN_0$ then $\abs{\sum_{j=0}^n {\alpha_{\boldsymbol{\sigma}'(j)}} - {\ell}_{\boldsymbol{\sigma}}} \leq \epsilon$. In other words, $\sum_{j=0}^{\infty} {\alpha_{\boldsymbol{\sigma}'(j)}} = \lim_{n \to \infty} {\sum_{j=0}^n {\alpha_{\boldsymbol{\sigma}'(j)}}} = {\ell}_{\boldsymbol{\sigma}}$, and therefore we have ${\ell}_{\boldsymbol{\sigma}} = {\ell}_{\boldsymbol{\sigma}'}$.  
\end{proof}

\subsection{Multiple Series of Complex-Valued Functions}

\begin{definition}[Unconditional pointwise and uniform convergence of multiple infinite series of functions]
Let $d \in \NN$ and $(f_{\mathbf{n}})_{\mathbf{n} \in \NN_0^d}$ be a sequence of complex-valued (multivariate) functions defined on a set $\mathcal{S}$. We say that the infinite series $\sum_{\mathbf{n} \in \NN_0^d} {f_{\mathbf{n}}(\mathbf{z})}$ is \emph{unconditionally pointwise (resp. uniformly) convergent} on $\mathcal{S}$ if, and only if, there exists a function $\overline{f} \colon \mathcal{S} \to \CC$ such that, for every bijection $\boldsymbol{\sigma} \colon \NN_0 \to \NN_0^d$, the ordinary series $\sum_{j=0}^{\infty} {f_{\boldsymbol{\sigma}(j)}(\mathbf{z})}$ converges pointwise (resp. uniformly) on $\mathcal{S}$, in the usual sense, to $\overline{f}(\mathbf{z})$. The unique function $\overline{f}$, which is independent of rearrangements of the sequence $(f_{\mathbf{n}})_{\mathbf{n} \in \NN_0^d}$, is called the pointwise (resp. uniform) limiting function/sum of the series and we write $\sum_{\mathbf{n} \in \NN_0^d} {f_{\mathbf{n}}(\mathbf{z})} = \overline{f}(\mathbf{z})$ (resp. $\sum_{\mathbf{n} \in \NN_0^d} {f_{\mathbf{n}}(\mathbf{z})} \; \stackrel{\mathclap{\mathrm{unif.}}}{=} \; \overline{f}(\mathbf{z})$) on $\mathcal{S}$. 
\end{definition}

Note that for a single ($d=1$) infinite series $\sum_{n \in \NN_0} {f_n(\mathbf{z})}$, unconditional pointwise (resp. uniform) convergence implies standard pointwise (resp. uniform) convergence on the same set and to the same limiting function.

\begin{proposition}[Unconditional uniform convergence implies unconditional pointwise convergence] \label{prop:Unconditional_uniform_convergence_implies_unconditional_pointwise_convergence}
If a multiple series of complex-valued functions $\sum_{\mathbf{n} \in \NN_0^d} {f_{\mathbf{n}}(\mathbf{z})}$ is unconditionally uniformly convergent on a set $\mathcal{S}$ to a function $\overline{f}(\mathbf{z})$, then the series is unconditionally pointwise convergent on $\mathcal{S}$ to $\overline{f}(\mathbf{z})$. Symbolically, we can write $\sum_{\mathbf{n} \in \NN_0^d} {f_{\mathbf{n}}(\mathbf{z})} \; \stackrel{\mathclap{\mathrm{unif.}}}{=} \; \overline{f}(\mathbf{z})$ on $\mathcal{S}$ $\implies$ $\sum_{\mathbf{n} \in \NN_0^d} {f_{\mathbf{n}}(\mathbf{z})} = \overline{f}(\mathbf{z})$ on $\mathcal{S}$.
\end{proposition}

\begin{proof}
Direct consequence of the fact that (standard) uniform convergence implies (standard) pointwise convergence on the same set and to the same limiting function. 
\end{proof}

\begin{proposition}[Equivalent definition of unconditional pointwise convergence] \label{prop:Equivalent_definition_of_unconditional_pointwise_convergence}
The multiple series of complex-valued functions $\sum_{\mathbf{n} \in \NN_0^d} {f_{\mathbf{n}}(\mathbf{z})}$ is unconditionally pointwise convergent on $\mathcal{S}$ if and only if for every $\mathbf{z}_0 \in \mathcal{S}$ the multiple series of complex numbers $\sum_{\mathbf{n} \in \NN_0^d} {f_{\mathbf{n}}(\mathbf{z}_0)}$ is unconditionally convergent. 
\end{proposition}

\begin{proof}
If $\sum_{\mathbf{n} \in \NN_0^d} {f_{\mathbf{n}}(\mathbf{z})}$ is unconditionally pointwise convergent, then there is a function $\overline{f} \colon \mathcal{S} \to \CC$ such that, for every bijection $\boldsymbol{\sigma} \colon \NN_0 \to \NN_0^d$ and every $\mathbf{z}_0 \in \mathcal{S}$, we have $\sum_{j=0}^{\infty} {f_{\boldsymbol{\sigma}(j)}(\mathbf{z}_0)} = \overline{f}(\mathbf{z}_0)$. Hence, for every $\mathbf{z}_0 \in \mathcal{S}$ there is an $\ell \coloneqq \overline{f}(\mathbf{z}_0) \in \CC$ such that for every bijection $\boldsymbol{\sigma} \colon \NN_0 \to \NN_0^d$ it holds that $\sum_{j=0}^{\infty} {f_{\boldsymbol{\sigma}(j)}(\mathbf{z}_0)} = \ell$, i.e., the series $\sum_{\mathbf{n} \in \NN_0^d} {f_{\mathbf{n}}(\mathbf{z}_0)}$ is unconditionally convergent.

Conversely, suppose that for every $\mathbf{z}_0 \in \mathcal{S}$ there is an $\ell = \ell(\mathbf{z}_0) \in \CC$ such that for every bijection $\boldsymbol{\sigma} \colon \NN_0 \to \NN_0^d$ we have $\sum_{j=0}^{\infty} {f_{\boldsymbol{\sigma}(j)}(\mathbf{z}_0)} = \ell$. Let us choose the function $\overline{f}(\mathbf{z}) \coloneqq \ell(\mathbf{z})$, for all $\mathbf{z} \in \mathcal{S}$. Then, for every bijection $\boldsymbol{\sigma} \colon \NN_0 \to \NN_0^d$ and every $\mathbf{z}_0 \in \mathcal{S}$ it holds that $\sum_{j=0}^{\infty} {f_{\boldsymbol{\sigma}(j)}(\mathbf{z}_0)} = \overline{f}(\mathbf{z}_0)$, i.e., the series $\sum_{\mathbf{n} \in \NN_0^d} {f_{\mathbf{n}}(\mathbf{z})}$ is unconditionally pointwise convergent on $\mathcal{S}$. 
\end{proof}

A uniformly convergent series of functions can potentially be reordered into a non-uniformly convergent series. Therefore, we need sufficient conditions to rule out this possibility. In particular, unconditional uniform convergence can be achieved by generalizing the Weierstrass M-test.

\begin{lemma}[Classical Weierstrass M-test for single infinite series of functions {\cite[Theorem~7.10]{Rudin_1976}}] \label{lem:Classical_Weierstrass_M-test}
Let $(f_n)_{n \in \NN_0}$ be a sequence of complex-valued (multivariate) functions defined on a set $\mathcal{S}$. Assume that there exists a sequence of nonnegative numbers $(M_n)_{n \in \NN_0}$ such that: i) $\abs{f_n(\mathbf{z})} \leq M_n$, for all $n \in \NN_0$ and $\mathbf{z} \in \mathcal{S}$, and ii) the series $\sum_{n \in \NN_0} {M_n}$ is (absolutely) convergent. Then, the infinite series $\sum_{n \in \NN_0} {f_n(\mathbf{z})}$ is uniformly convergent on $\mathcal{S}$. 
\end{lemma}

\begin{proposition}[Generalized Weierstrass M-test for multiple infinite series of functions] \label{prop:Generalized_Weierstrass_M-test}
Let $d \in \NN$ and $(f_{\mathbf{n}})_{\mathbf{n} \in \NN_0^d}$ be a sequence of complex-valued (multivariate) functions defined on a set $\mathcal{S}$. Suppose that there exists a sequence of nonnegative numbers $(M_{\mathbf{n}})_{\mathbf{n} \in \NN_0^d}$ such that: i) $\abs{f_{\mathbf{n}}(\mathbf{z})} \leq M_{\mathbf{n}}$, for all $\mathbf{n} \in \NN_0^d$ and $\mathbf{z} \in \mathcal{S}$, and ii) the multiple series $\sum_{\mathbf{n} \in \NN_0^d} {M_{\mathbf{n}}}$ is absolutely convergent. Then, the infinite series $\sum_{\mathbf{n} \in \NN_0^d} {f_{\mathbf{n}}(\mathbf{z})}$ is unconditionally uniformly convergent on $\mathcal{S}$. 
\end{proposition}

\begin{proof}
Let $\boldsymbol{\sigma} \colon \NN_0 \to \NN_0^d$ be a particular, but arbitrarily chosen, bijection. From the assumptions, it holds that $\abs{f_{\boldsymbol{\sigma}(j)}(\mathbf{z})} \leq M_{\boldsymbol{\sigma}(j)}$ for all $j \in \NN_0$ and $\mathbf{z} \in \mathcal{S}$, and $\sum_{j=0}^{\infty} {M_{\boldsymbol{\sigma}(j)}} = \sum_{j=0}^{\infty} \abs{M_{\boldsymbol{\sigma}(j)}} \leq \sum_{\mathbf{n} \in \NN_0^d} \abs{M_{\mathbf{n}}} < \infty$. By the classical Weierstrass M-test (Lemma~\ref{lem:Classical_Weierstrass_M-test}), we know that the ordinary series $\sum_{j=0}^{\infty} {f_{\boldsymbol{\sigma}(j)}(\mathbf{z})}$ converges uniformly on $\mathcal{S}$ to a function $\overline{f}_{\boldsymbol{\sigma}} \colon \mathcal{S} \to \CC$. Now, we have to prove that the limiting function $\overline{f}_{\boldsymbol{\sigma}}$ is independent of $\boldsymbol{\sigma}$. Equivalently, we will show that $\overline{f}_{\boldsymbol{\sigma}}(\mathbf{z}) = \overline{f}_{\boldsymbol{\sigma}'}(\mathbf{z})$ for all $\mathbf{z} \in \mathcal{S}$ and all bijections $\boldsymbol{\sigma}$ and $\boldsymbol{\sigma}'$ (from $\NN_0$ to $\NN_0^d$). For this purpose let $\mathbf{z}_0 \in \mathcal{S}$ be fixed, and $\boldsymbol{\sigma}, \boldsymbol{\sigma}'$ be two such bijections with associated limiting functions $\overline{f}_{\boldsymbol{\sigma}}$ and $\overline{f}_{\boldsymbol{\sigma}'}$, respectively. It holds that $\sum_{\mathbf{n} \in \NN_0^d} \abs{f_{\mathbf{n}}(\mathbf{z}_0)} \leq \sum_{\mathbf{n} \in \NN_0^d} \abs{M_{\mathbf{n}}} < \infty$, hence the series $\sum_{\mathbf{n} \in \NN_0^d} {f_{\mathbf{n}}(\mathbf{z}_0)}$ is absolutely convergent and, by Proposition~\ref{prop:Absolute_convergence_implies_unconditional_convergence}, unconditionally convergent as well. This means that there is an $\ell_0 \in \CC$ such that $\sum_{j=0}^{\infty} {f_{\boldsymbol{\sigma}(j)}(\mathbf{z}_0)} = \sum_{j=0}^{\infty} {f_{\boldsymbol{\sigma}'(j)}(\mathbf{z}_0)} = \ell_0$. In addition, uniform convergence implies pointwise convergence to the same limiting function, therefore $\overline{f}_{\boldsymbol{\sigma}}(\mathbf{z}_0) = \sum_{j=0}^{\infty} {f_{\boldsymbol{\sigma}(j)}(\mathbf{z}_0)} = \ell_0$ and $\overline{f}_{\boldsymbol{\sigma}'}(\mathbf{z}_0) = \sum_{j=0}^{\infty} {f_{\boldsymbol{\sigma}'(j)}(\mathbf{z}_0)} = \ell_0$. As a consequence, it follows that $\overline{f}_{\boldsymbol{\sigma}}(\mathbf{z}_0) = \overline{f}_{\boldsymbol{\sigma}'}(\mathbf{z}_0)$.  
\end{proof}

\subsection{Power Series in Several Complex Variables}

\begin{proposition}[Sufficient conditions for unconditional uniform and pointwise convergence of multiple power series] \label{prop:Sufficient_condition_for_unconditional_uniform_and_pointwise_convergence_of_multiple_power_series}
Let $d \in \NN$ and $\sum_{\mathbf{n} \in \NN_0^d} {\alpha_{\mathbf{n}} {\mathbf{z}}^{\mathbf{n}}}$ be a power series in several complex variables $\mathbf{z} \in \CC^d$ with complex coefficients $\{\alpha_{\mathbf{n}}\}$. Suppose that there exist a constant $C \in \RR_{+}$ and functions $h_j \colon \NN_0 \to \RR_{+}$, $j \in \{1,\dots,d\}$, such that: $\abs{\alpha_{\mathbf{n}}} \leq C \prod_{j=1}^d {h_j(n_j)}$, for all $\mathbf{n} \in \NN_0^d$, and $\ell_j \coloneqq \left(\limsup_{n \to \infty}  \left(h_j(n)\right)^{1/n} \right)^{-1} > 0$, for all $j \in \{1,\dots,d\}$. Then, the multiple power series $\sum_{\mathbf{n} \in \NN_0^d} {\alpha_{\mathbf{n}} {\mathbf{z}}^{\mathbf{n}}}$ is unconditionally uniformly convergent on the poly-disk $\overline{\mathcal{D}}(\boldsymbol{\rho})$, where $\boldsymbol{\rho} = [\rho_1,\dots,\rho_d]^{\top} \in \RR_+^d$ is any poly-radius such that $\rho_j < \ell_j$ for all $j \in \{1,\dots,d\}$. In addition, the series $\sum_{\mathbf{n} \in \NN_0^d} {\alpha_{\mathbf{n}} {\mathbf{z}}^{\mathbf{n}}}$ is unconditionally pointwise convergent on the open poly-disk $\mathcal{D}(\boldsymbol{\ell})$, where $\boldsymbol{\ell} = [\ell_1,\dots,\ell_d]^{\top}$.
\end{proposition}

\begin{proof}
Let $\boldsymbol{\rho}$ be an arbitrary poly-radius such that $0 < \rho_j < \ell_j$ for all $j \in \{1,\dots,d\}$. Then, for all $\mathbf{n} \in \NN_0^d$ and $\mathbf{z} \in \overline{\mathcal{D}}(\boldsymbol{\rho})$, we obtain 
\begin{equation*}
\abs{\alpha_{\mathbf{n}} {\mathbf{z}}^{\mathbf{n}}} = \abs{\alpha_{\mathbf{n}}} \prod_{j=1}^d {\abs{z_j}^{n_j}} \leq C \prod_{j=1}^d {h_j(n_j) \abs{z_j}^{n_j}}  \leq  C \prod_{j=1}^d {h_j(n_j) {\rho_j^{n_j}}}  \eqqcolon  M_{\mathbf{n}}  . 
\end{equation*}
The multiple series $\sum_{\mathbf{n} \in \NN_0^d} {M_{\mathbf{n}}}$ is absolutely convergent, because 
\begin{align*} 
\sum_{\mathbf{n} \in \NN_0^d} \abs{M_{\mathbf{n}}} & = C \sup \underbrace{ \left\{ \sum_{\mathbf{n} \in \mathcal{F}} \prod_{j=1}^d {h_j(n_j) {\rho_j^{n_j}}} : \mathcal{F} \subseteq \NN_0^d,\, \mathcal{F}\ \mathrm{is\ finite} \right\} }_{\mathcal{A} \coloneqq}  \\
& = C \sup \underbrace{ \left\{ \sum_{\mathbf{n} \in \prod_{j=1}^d {\mathcal{F}_j}} \prod_{j=1}^d {h_j(n_j) {\rho_j^{n_j}}} :  \mathcal{F}_j \subseteq \NN_0,\, \mathcal{F}_j\ \mathrm{is\ finite},\, \forall j \in \{1,\dots,d\} \right\} }_{\mathcal{B} \coloneqq}  \\
& = C \sup \left\{ \prod_{j=1}^d \left( \sum_{n \in {\mathcal{F}_j}} {h_j(n) \rho_j^{n}} \right) :  \mathcal{F}_j \subseteq \NN_0,\, \mathcal{F}_j\ \mathrm{is\ finite},\, \forall j \in \{1,\dots,d\} \right\}     \\
& = C \prod_{j=1}^d \left( \sum_{n \in \NN_0} {h_j(n) \rho_j^{n}} \right) < \infty .
\end{align*} 
Here, the second equality (i.e., $\sup \mathcal{A} = \sup \mathcal{B}$) is true due to the following reasons: 1) the Cartesian product of finitely many finite sets is finite, so $\mathcal{B} \subseteq \mathcal{A}$ and therefore $\sup \mathcal{B} \leq \sup \mathcal{A}$, and 2) given any finite set $\mathcal{F}$ there are finite sets $\{\mathcal{F}_j\}_{j=1}^d$ whose  Cartesian product $\mathcal{G} \coloneqq \prod_{j=1}^d {\mathcal{F}_j} \supseteq \mathcal{F}$, so $\sum_{\mathbf{n} \in \mathcal{F}} (\cdot) \leq \sum_{\mathbf{n} \in \mathcal{G}} (\cdot)$ and hence $\sup \mathcal{A} \leq \sup \mathcal{B}$. Moreover, observe that the ordinary series $\sum_{n \in \NN_0} {h_j(n) \rho_j^{n}}$ is (absolutely) convergent, for all $j \in \{1,\dots,d\}$, since by the root test we have $\limsup_{n \to \infty} \abs{h_j(n) \rho_j^{n}}^{1/n} = \rho_j \ell_j^{-1} < 1$.  Now, the unconditional uniform convergence of the power series $\sum_{\mathbf{n} \in \NN_0^d} {\alpha_{\mathbf{n}} {\mathbf{z}}^{\mathbf{n}}}$ on the poly-disk $\overline{\mathcal{D}}(\boldsymbol{\rho})$ follows from the generalized Weierstrass M-test (Proposition~\ref{prop:Generalized_Weierstrass_M-test}). 

Finally, given any $\mathbf{z}_0 \in \mathcal{D}(\boldsymbol{\ell})$ the series $\sum_{\mathbf{n} \in \NN_0^d} {\alpha_{\mathbf{n}} {\mathbf{z}}^{\mathbf{n}}}$ is unconditionally uniformly convergent on $\overline{\mathcal{D}}(\abs{z_{0,1}}+\epsilon_1,\dots,\abs{z_{0,d}}+\epsilon_d)$ to a limiting function $\overline{f}(\mathbf{z})$, where the positive numbers $\{\epsilon_j\}$ are chosen sufficiently small so that $\abs{z_{0,j}} + \epsilon_j < \ell_j$ for all $j \in \{1,\dots,d\}$; for example, we can choose $\epsilon_j = (\ell_j - \abs{z_{0,j}})/2$. By Proposition~\ref{prop:Unconditional_uniform_convergence_implies_unconditional_pointwise_convergence}, we obtain $\sum_{\mathbf{n} \in \NN_0^d} {\alpha_{\mathbf{n}} {\mathbf{z}}^{\mathbf{n}}} = \overline{f}(\mathbf{z})$ for all $\mathbf{z} \in \overline{\mathcal{D}}(\abs{z_{0,1}}+\epsilon_1,\dots,\abs{z_{0,d}}+\epsilon_d)$, and therefore $\sum_{\mathbf{n} \in \NN_0^d} {\alpha_{\mathbf{n}} {\mathbf{z}}_0^{\mathbf{n}}} = \overline{f}(\mathbf{z}_0)$. In other words, we have shown that for every $\mathbf{z}_0 \in \mathcal{D}(\boldsymbol{\ell})$ the series $\sum_{\mathbf{n} \in \NN_0^d} {\alpha_{\mathbf{n}} {\mathbf{z}}_0^{\mathbf{n}}}$ is unconditionally convergent. Hence, by Proposition~\ref{prop:Equivalent_definition_of_unconditional_pointwise_convergence}, the series $\sum_{\mathbf{n} \in \NN_0^d} {\alpha_{\mathbf{n}} {\mathbf{z}}^{\mathbf{n}}}$ is unconditionally pointwise convergent on $\mathcal{D}(\boldsymbol{\ell})$. 
\end{proof}

Proposition~\ref{prop:Sufficient_condition_for_unconditional_uniform_and_pointwise_convergence_of_multiple_power_series} generalizes a well-known result for $d=1$: every single power series $\sum_{n=0}^{\infty} {\alpha_n z^n}$ with radius of convergence $R \coloneqq \left(\limsup_{n \to \infty} \abs{\alpha_n}^{1/n} \right)^{-1} > 0$ is pointwise convergent for $\abs{z} < R$ and uniformly convergent for $\abs{z} \leq \rho$, where $\rho$ is any positive number less than $R$ \cite[\S 1.21]{Titchmarsh_1939}.

\section{Useful Lemmas} \label{sec:Useful_Lemmas}

Next, we provide several helpful lemmas that are required to prove Theorems~\ref{thrm:Theorem_1}~and~\ref{thrm:Theorem_2}.

\subsection{Identities for Sums of Iverson's Brackets}

\begin{lemma}[Iverson-bracket sum identities] \label{lem:Iverson-bracket_sum_identities}
Let $m \in \NN$. Then, the following identities involving the Iverson bracket hold:
\begin{equation} \label{eq:Sum_q_r_Iverson_bracket_identity}
\sum_{q=0}^{\infty} \sum_{r=0}^{m-1} \iverson{n = q m + r} = 1 , \quad \forall n \in \NN_0 ,
\end{equation}
\begin{equation} \label{eq:Sum_n_Iverson_bracket_identity}
\sum_{n=0}^{\infty} \iverson{n = q m + r} = 1   , \quad \forall q \in \NN_0 , \ \forall r \in {\mathcal{R}(m)}  .
\end{equation}
\end{lemma}

\begin{proof}
The former identity is an immediate consequence of the quotient-remainder theorem: given any $m \in \NN$ and $n \in \NN_0$, there exist unique integers $q \in \NN_0$ and $r \in {\mathcal{R}(m)}$ such that $n = q m + r$; in particular, $q = \lfloor n/m \rfloor$ and $r = n \bmod m$. The latter identity follows from the simple fact that, given any $m \in \NN$, $q \in \NN_0$ and $r \in {\mathcal{R}(m)}$, there exists a unique integer $n \in \NN_0$ such that $n = q m + r$. 
\end{proof}

\begin{lemma}[Multi-dimensional Iverson-bracket sum identities] \label{lem:Multi-dimensional_Iverson-bracket_sum_identities}
Let $\mathbf{m} \in \NN^d$. Then, we have the following identities involving the Iverson bracket: 
\begin{equation} \label{eq:Sum_q_r_Iverson_bracket_identity_multi-dimensional}
\sum_{\mathbf{q} \in \NN_0^d} \sum_{\mathbf{r} \in {\mathcal{R}(\mathbf{m})}} \iverson{\mathbf{n} = \mathbf{q} \odot \mathbf{m} + \mathbf{r}} = 1 , \quad \forall \mathbf{n} \in \NN_0^d .
\end{equation}
\begin{equation} \label{eq:Sum_n_Iverson_bracket_identity_multi-dimensional}
\sum_{\mathbf{n} \in \NN_0^d} \iverson{\mathbf{n} = \mathbf{q} \odot \mathbf{m} + \mathbf{r}} = 1   , \quad \forall \mathbf{q} \in \NN_0^d , \ \forall \mathbf{r} \in {\mathcal{R}(\mathbf{m})}  .
\end{equation}
\end{lemma}

\begin{proof}
By a straightforward generalization of Lemma~\ref{lem:Iverson-bracket_sum_identities}. Specifically, is holds that $\iverson{\mathbf{n} = \mathbf{q} \odot \mathbf{m} + \mathbf{r}} = \iverson{n_j = q_j m_j + r_j, \; \forall j \in \{1,\dots,d\}} = \prod_{j=1}^d \iverson{n_j = q_j m_j + r_j}$. Therefore,
\begin{equation*}
\sum_{\mathbf{q} \in \NN_0^d} \sum_{\mathbf{r} \in {\mathcal{R}(\mathbf{m})}} \iverson{\mathbf{n} = \mathbf{q} \odot \mathbf{m} + \mathbf{r}} =  \prod_{j=1}^d \left( \sum_{q \in \NN_0} \sum_{r \in {\mathcal{R}(m_j)}} \iverson{n_j = q m_j + r} \right) \stackrel{\mathclap{\eqref{eq:Sum_q_r_Iverson_bracket_identity}}}{=} \prod_{j=1}^d {1} = 1 ,
\end{equation*}
for all $\mathbf{n} \in \NN_0^d$, and 
\begin{equation*}
\sum_{\mathbf{n} \in \NN_0^d} \iverson{\mathbf{n} = \mathbf{q} \odot \mathbf{m} + \mathbf{r}} = \prod_{j=1}^d \left( \sum_{n \in \NN_0} \iverson{n = q_j m_j + r_j}  \right) \stackrel{\mathclap{\eqref{eq:Sum_n_Iverson_bracket_identity}}}{=} \prod_{j=1}^d {1} = 1    ,
\end{equation*} 
for all $\mathbf{q} \in \NN_0^d$ and $\mathbf{r} \in {\mathcal{R}(\mathbf{m})}$. 
\end{proof}

\subsection{Integral Representations of the Kronecker Deltas}

\begin{lemma}[Integral representation of the Kronecker delta] \label{lem:Integral_representation_Kronecker_delta}
Let $n,k \in \ZZ$ and $\rho \in \RR_{+}$. Then, we have
\begin{equation} \label{eq:Kronecker_delta_integral_representation}
\delta_{n,k} \coloneqq \iverson{n = k} = \frac{1}{2\pi\iu} \cintpos_{\mathcal{C}(\rho)} {u^{n-k-1}} \rd u .
\end{equation}
\end{lemma}

\begin{proof}
By applying the variable transformation $u \mapsto \rho e^{\iu\theta}$, $0 \leq \theta < 2\pi$, we obtain 
\begin{equation*}
\frac{1}{2\pi\iu} \cintpos_{\mathcal{C}(\rho)} {u^{n-k-1}} \rd u = \frac{\rho^{n-k}}{2\pi} \int_0^{2\pi} e^{\iu(n-k)\theta}  \rd \theta = 
\begin{cases}
1 , & \mathrm{if}\ n = k  , \\
0 , & \mathrm{if}\ n \neq k  . 
\end{cases} 
= \iverson{n = k} .
\end{equation*}
\end{proof}

\begin{lemma}[Integral representation of the multi-dimensional Kronecker delta] \label{lem:Integral_representation_multi-dimensional_Kronecker_delta}
Let $\mathbf{n},\mathbf{k} \in \ZZ^d$ and $\boldsymbol{\rho} \in \RR_{+}^d$. Then, we have 
\begin{equation} \label{eq:Kronecker_delta_integral_representation_multi-dimensional}
\delta_{\mathbf{n},\mathbf{k}} \coloneqq \iverson{\mathbf{n} = \mathbf{k}} = \frac{1}{(2\pi\iu)^d} \cintpos_{\mathcal{C}(\boldsymbol{\rho})} {\mathbf{u}^{\mathbf{n} - \mathbf{k} - \mathbf{1}} \rd \mathbf{u}} .
\end{equation}
\end{lemma}

\begin{proof}
By a direct generalization of Lemma~\ref{lem:Integral_representation_Kronecker_delta}. In particular, it holds that  
\begin{align*}
\iverson{\mathbf{n} = \mathbf{k}} & = \iverson{n_j = k_j, \; \forall j \in \{1,\dots,d\}} = \prod_{j=1}^d \iverson{n_j = k_j} \\
& \stackrel{\mathclap{\eqref{eq:Kronecker_delta_integral_representation}}}{=}  \prod_{j=1}^d \left( \frac{1}{2\pi\iu} \cintpos_{\mathcal{C}(\rho_j)} {u^{n_j-k_j-1}} \rd u \right)  \\
& = \frac{1}{(2\pi\iu)^d} \cintpos_{u_1 \in \mathcal{C}(\rho_1)} \cdots \cintpos_{u_d \in \mathcal{C}(\rho_d)} {\left( \prod _{j=1}^d {u_j^{n_j-k_j-1}} \right)  \left(  \prod_{j=1}^d \mathrm{d} u_j \right)}    \\
& = \frac{1}{(2\pi\iu)^d} \cintpos_{\mathcal{C}(\boldsymbol{\rho})} {\mathbf{u}^{\mathbf{n} - \mathbf{k} - \mathbf{1}} \rd \mathbf{u}}    .
\end{align*} 
\end{proof}

\subsection{Partial Fraction Decomposition}

Another useful result to facilitate the calculation of residues is the partial fraction decomposition of a rational function. It can simplify integration by integrating each term of the expression separately.

\begin{lemma} [Partial fraction decomposition of a rational function, based on {\cite[Theorems~4.2--4.5]{Lang_2005}}] \label{lem:Partial_fraction_decomposition}
Let $f,g \colon \CC \to \CC$ be relatively prime (or, coprime) polynomials, i.e., without common roots, with $\operatorname{deg}(f) < \operatorname{deg}(g)$. Assume also that $g(u) = \prod_{j=1}^k (u - u_j)^{n_j}$ for some $k \in \NN$, where $\{u_j\}$ are distinct complex numbers, and $\{n_j\}$ are (strictly) positive integers. Then, there exist unique complex numbers $\{c_{j,l}\}$ such that 
\begin{equation*}
R(u) \coloneqq \frac{f(u)}{g(u)} = \sum_{j=1}^k \sum_{l=1}^{n_j} \frac{c_{j,l}}{(u-u_j)^l}  ,
\end{equation*} 
for all $u \in \CC \setminus \{u_1,\dots,u_k\}$.
\end{lemma}

\subsection{Cauchy's Integral Formulas} 

Finally, residues of a meromorphic function with general-order \emph{independent} poles, i.e., isolated singularities, can be computed using the Cauchy integral formulas.

\begin{lemma}[Cauchy's integral formulas on poly-disks for multiple power series] \label{lem:Cauchy_integral_formulas}
Let $d \in \NN$ and $\boldsymbol{\rho} \in \RR_{+}^d$. Suppose that the multiple power series $\sum_{\mathbf{n} \in \NN_0^d} a_{\mathbf{n}} \mathbf{z}^{\mathbf{n}}$, with $a_{\mathbf{n}} \in \CC$ and $\mathbf{z} \in \CC^d$, is unconditionally uniformly convergent on the poly-disk $\overline{\mathcal{D}} (\boldsymbol{\rho}) = \mathcal{D} (\boldsymbol{\rho}) \cup \mathcal{C} (\boldsymbol{\rho})$ to a function $f \colon \overline{\mathcal{D}} (\boldsymbol{\rho}) \to \CC$; we can write  $f(\mathbf{z}) = \sum_{\mathbf{n} \in \NN_0^d} a_{\mathbf{n}} \mathbf{z}^{\mathbf{n}}$. Then, $f$ is continuous on its domain $\overline{\mathcal{D}} (\boldsymbol{\rho})$, holomorphic\footnote{We say that a function $f \colon \mathcal{S} \to \CC$ is holomorphic on an open set $\mathcal{S} \subseteq \CC^d$ (for some $d \in \NN$) if, and only if, it is holomorphic in each variable separately, i.e., for each $j \in \{1,\dots,d\}$ and each $\mathbf{z} \in \mathcal{S}$ the function $g_{j,\mathbf{z}}(w) \coloneqq f(z_1,\dots,z_{j-1},w,z_{j+1},\dots,z_d)$ is holomorphic, in the classical one-variable sense, on the set $\mathcal{W}_{j,\mathbf{z}} \coloneqq \{w \in \CC : (z_1,\dots,z_{j-1},w,z_{j+1},\dots,z_d) \in \mathcal{S}\}$. In particular, if $\mathcal{S} = \mathcal{D} (\boldsymbol{\rho})$ for some $\boldsymbol{\rho} \in \RR_{+}^d$, then every set $\mathcal{W}_{j,\mathbf{z}} = \mathcal{W}_j = \mathcal{D}(\rho_j)$.} on the open poly-disk $\mathcal{D} (\boldsymbol{\rho})$, and the complex coefficients $\{a_{\mathbf{n}}\}_{\mathbf{n} \in \NN_0^d}$ are given by  
\begin{equation} \label{eq:Coefficients_a_n}
a_{\mathbf{n}}  =  \frac{f^{(\mathbf{n})}(\mathbf{0})}{\mathbf{n}!}  =  \frac{1}{(2\pi\iu)^d} \cintpos_{\mathcal{C}(\boldsymbol{\rho})} \frac{f(\mathbf{z})}{\mathbf{z}^{\mathbf{n}+\mathbf{1}}} \rd \mathbf{z}  .
\end{equation}
In addition, for every $\mathbf{n} \in \NN_0^d$, $f^{(\mathbf{n})}$ is holomorphic on $\mathcal{D} (\boldsymbol{\rho})$ with 
\begin{equation} \label{eq:Residue_with_independent_poles}
\frac{f^{(\mathbf{n})}(\boldsymbol{\xi})}{\mathbf{n}!}  =  \frac{1}{(2\pi\iu)^d} \cintpos_{\mathcal{C}(\boldsymbol{\rho})}  \frac{f(\mathbf{z})}{(\mathbf{z}-\boldsymbol{\xi})^{\mathbf{n}+\mathbf{1}}} \rd \mathbf{z}  ,
\end{equation}
for all $\boldsymbol{\xi} \in \mathcal{D}(\boldsymbol{\rho})$, and the multiple series of derivatives
\begin{equation} \label{eq:Uniform_convergence_of_series_of_derivatives}
\sum_{\boldsymbol{\nu} \in \NN_0^d} a_{\boldsymbol{\nu}} (\mathbf{z}^{\boldsymbol{\nu}})^{(\mathbf{n})} \; \stackrel{\mathclap{\mathrm{unif.}}}{=} \; f^{(\mathbf{n})}(\mathbf{z})  \quad \text{on}\ \mathcal{D}(\boldsymbol{\rho})  .
\end{equation}
\end{lemma}

\begin{proof}
According to \cite[Theorem~7.12]{Rudin_1976}, if $\{g_n\}$ is a sequence of continuous functions on a set $\mathcal{S}$ and this sequence converges uniformly on $\mathcal{S}$ to a function $g$, then $g$ is continuous on $\mathcal{S}$. Since all (finite) partial sums of the multiple series $\sum_{\mathbf{n} \in \NN_0^d} a_{\mathbf{n}} \mathbf{z}^{\mathbf{n}}$ are continuous on $\overline{\mathcal{D}} (\boldsymbol{\rho})$ and uniformly convergent to $f$, given any rearrangement of the series' terms, we conclude that $f$ is continuous on $\overline{\mathcal{D}} (\boldsymbol{\rho})$. 

Due to Proposition~\ref{prop:Unconditional_uniform_convergence_implies_unconditional_pointwise_convergence}, we have that the series $\sum_{\mathbf{n} \in \NN_0^d} a_{\mathbf{n}} \mathbf{z}^{\mathbf{n}}$ is unconditionally pointwise convergent on $\overline{\mathcal{D}} (\boldsymbol{\rho})$ (and therefore on its subset $\mathcal{D} (\boldsymbol{\rho})$) to $f$. By \cite[Theorem~1.17]{Range_1986}, it follows that $f$ is holomorphic on the open poly-disk $\mathcal{D} (\boldsymbol{\rho})$, Eq.~\eqref{eq:Uniform_convergence_of_series_of_derivatives} is true, and $a_{\mathbf{n}} = \frac{f^{(\mathbf{n})}(\mathbf{0})}{\mathbf{n}!}$ for all $\mathbf{n} \in \NN_0^d$. 

Moreover, for every $\mathbf{n} \in \NN_0^d$, $f^{(\mathbf{n})}$ is holomorphic on $\mathcal{D} (\boldsymbol{\rho})$ and Eq.~\eqref{eq:Residue_with_independent_poles} holds for all $\boldsymbol{\xi} \in \mathcal{D}(\boldsymbol{\rho})$, because $f$ is continuous on $\overline{\mathcal{D}} (\boldsymbol{\rho})$ and holomorphic on $\mathcal{D} (\boldsymbol{\rho})$; see \cite[Theorem~4.1, Corollary~4.2]{Aizenberg-Yuzhakov_1983} and \cite[Theorem~1.3, Corollary~1.5]{Range_1986}. Finally, the second equality in \eqref{eq:Coefficients_a_n} can be obtained from \eqref{eq:Residue_with_independent_poles} by setting $\boldsymbol{\xi} = \mathbf{0}$. 
\end{proof}

\section{Proof of Proposition~\ref{prop:Unconditional_uniform_and_pointwise_convergence_of_GFs}} \label{sec:Proof_Proposition_1}

\begin{lemma}[Rough upper bound on $t_k(n)$] \label{lem:Rough_upper_bound_on_t_k}
It holds that  
\begin{equation*}
t_k(n) \leq \begin{cases}
\maxnorm{\mathbf{a}}^k (n + \chi) - \chi , & \mathrm{if}\ \maxnorm{\mathbf{a}} \neq 1  , \\
n + \maxnorm{\mathbf{b}} k , & \mathrm{if}\ \maxnorm{\mathbf{a}} = 1  ,
\end{cases}
\end{equation*}
for all $(k,n) \in \NN_0^2$, where $\chi \coloneqq \frac{\maxnorm{\mathbf{b}}}{\maxnorm{\mathbf{a}} - 1}$.
\end{lemma}

\begin{proof}
According to \eqref{eq:Generalized_Collatz_dynamics} and by applying the triangle inequality in \eqref{eq:Generalized_Collatz_map}, we have
\begin{equation} \label{eq:Triangle_inequality_for_t_k}
t_k(n) = \abs{t_k(n)} = \abs{t(t_{k-1}(n))} \leq \maxnorm{\mathbf{a}} t_{k-1}(n) + \maxnorm{\mathbf{b}} ,  
\end{equation}
for all $k,n \in \NN_0$. Next, we proceed by induction on $k$ (given a particular, but arbitrary, $n \in \NN_0$) to prove that 
\begin{equation} \label{eq:Inequality_by_induction_on_k}
t_k(n) \leq \maxnorm{\mathbf{a}}^k n + \maxnorm{\mathbf{b}} \sum_{j=0}^{k-1} {\maxnorm{\mathbf{a}}^j} .
\end{equation}
\emph{Basis step}: For $k=0$, we have $t_0(n) = n \leq \maxnorm{\mathbf{a}}^0 n + \maxnorm{\mathbf{b}} \sum_{j=0}^{-1} {\maxnorm{\mathbf{a}}^j} = n$. \emph{Inductive step}: Suppose that \eqref{eq:Inequality_by_induction_on_k} is true for $k-1$ (\emph{inductive hypothesis}), and then prove that it is also true for $k$, whenever $k \geq 1$. In particular, by combining \eqref{eq:Triangle_inequality_for_t_k} and \eqref{eq:Inequality_by_induction_on_k} with $k \mapsto k-1$, we get
\begin{align*} 
t_k(n) & \leq \maxnorm{\mathbf{a}} \left( \maxnorm{\mathbf{a}}^{k-1} n + \maxnorm{\mathbf{b}} \sum_{j=0}^{k-2} {\maxnorm{\mathbf{a}}^j} \right) + \maxnorm{\mathbf{b}}  \\
& = \maxnorm{\mathbf{a}}^k n + \maxnorm{\mathbf{b}} \left( 1 + \sum_{j=0}^{k-2} {\maxnorm{\mathbf{a}}^{j+1}} \right)      \\
& = \maxnorm{\mathbf{a}}^k n + \maxnorm{\mathbf{b}} \left( 1 + \sum_{j=1}^{k-1} {\maxnorm{\mathbf{a}}^j} \right)  \\
& = \maxnorm{\mathbf{a}}^k n + \maxnorm{\mathbf{b}} \sum_{j=0}^{k-1} {\maxnorm{\mathbf{a}}^j}    . 
\end{align*} 
Hence, we have proved \eqref{eq:Inequality_by_induction_on_k} for all $k \in \NN_0$ (and all $n \in \NN_0$, since $n$ has been chosen arbitrarily). Finally, \eqref{eq:Inequality_by_induction_on_k} leads directly to the desired result, because the geometric series
\begin{equation*}
\sum_{j=0}^{k-1} {\maxnorm{\mathbf{a}}^j} = \begin{cases} 
\frac{\maxnorm{\mathbf{a}}^k - 1}{\maxnorm{\mathbf{a}} - 1} , & \mathrm{if}\ \maxnorm{\mathbf{a}} \neq 1  , \\
k , & \mathrm{if}\ \maxnorm{\mathbf{a}} = 1  .    
\end{cases}
\end{equation*}
\end{proof}

\begin{remark}
The shortened $3n+1$ map in \eqref{eq:Shortened_3n+1_map} can be obtained from the $(m,\mathbf{a},\mathbf{b})$-Collatz map, as a special case, by setting $m=2$, $\mathbf{a}=\left[\frac{1}{2},\frac{3}{2}\right]^{\top}$ and $\mathbf{b}=\left[0,\frac{1}{2}\right]^{\top}$. In addition, from Lemma~\ref{lem:Rough_upper_bound_on_t_k} with $\maxnorm{\mathbf{a}} = \frac{3}{2}$, $\maxnorm{\mathbf{b}} = \frac{1}{2}$ and $\chi = 1$, we obtain $t_k(n) \leq \left( \frac{3}{2} \right)^k (n+1) - 1$, which is consistent with \cite[Eq.~(12)]{Berg-Meinardus_1994}.
\end{remark}

By virtue of Lemma~\ref{lem:Rough_upper_bound_on_t_k} and the triangle inequality, we get   
\begin{align*} 
\abs{t_k(n)} = t_k(n) & \leq \begin{cases} 
\abs{\maxnorm{\mathbf{a}}^k n + \maxnorm{\mathbf{a}}^k \chi - \chi} , & \mathrm{if}\ \maxnorm{\mathbf{a}} \neq 1 , \\
n + \maxnorm{\mathbf{b}} k , & \mathrm{if}\ \maxnorm{\mathbf{a}} = 1  ,
\end{cases}   \\
& \leq \begin{cases} 
\maxnorm{\mathbf{a}}^k n + \maxnorm{\mathbf{a}}^k \abs{\chi} + \abs{\chi} , & \mathrm{if}\ \maxnorm{\mathbf{a}} \neq 1  , \\
n + \maxnorm{\mathbf{b}} k , & \mathrm{if}\ \maxnorm{\mathbf{a}} = 1  ,
\end{cases}   \\
& \leq \begin{cases} 
(\abs{\chi} + 1) (\maxnorm{\mathbf{a}}^k + 1) (n + 1) , & \mathrm{if}\ \maxnorm{\mathbf{a}} \neq 1  , \\
(\maxnorm{\mathbf{b}} k + 1) (n + 1) , & \mathrm{if}\ \maxnorm{\mathbf{a}} = 1  .
\end{cases}  
\end{align*} 
Given a fixed but arbitrary $k \in \NN_0$, let $C \coloneqq (\abs{\chi} + 1) (\maxnorm{\mathbf{a}}^k + 1)$ when $\maxnorm{\mathbf{a}} \neq 1$, or $C \coloneqq \maxnorm{\mathbf{b}} k + 1$ when $\maxnorm{\mathbf{a}} = 1$, and $h(n) \coloneqq n+1$. Since $\ell \coloneqq \left(\limsup_{n \to \infty}  \left(h(n)\right)^{1/n} \right)^{-1} = \left(\lim_{n \to \infty}  e^{\log(n+1)/n} \right)^{-1} = 1 > 0$, Proposition~\ref{prop:Sufficient_condition_for_unconditional_uniform_and_pointwise_convergence_of_multiple_power_series} implies the desired result for the power series $\sum_{n=0}^{\infty} t_k(n) z^n$. 

In addition, for the power series $\sum_{k=0}^{\infty} \sum_{n=0}^{\infty}  t_k(n) z^n w^k$ let us choose $h_1(n) \coloneqq n+1$, so $\ell_1 \coloneqq \left(\limsup_{n \to \infty}  \left(h_1(n)\right)^{1/n} \right)^{-1} = 1 > 0$, and consider two cases.
\begin{itemize}
\item \emph{Case 1} (when $\maxnorm{\mathbf{a}} \neq 1$): Let $C \coloneqq (\abs{\chi} + 1)$ and $h_2(k) \coloneqq \maxnorm{\mathbf{a}}^k + 1$. Then, for $\ell_2 \coloneqq \left(\limsup_{k \to \infty}  \left(h_2(k)\right)^{1/k} \right)^{-1}$ we should distinguish two subcases. 
\begin{itemize}
\item \emph{Subcase 1.a} (when $0 < \maxnorm{\mathbf{a}} < 1$): \newline $\ell_2 = \left(\lim_{k \to \infty}  (\maxnorm{\mathbf{a}}^k + 1)^{1/k} \right)^{-1} = 1 > 0$.

\item \emph{Subcase 1.b} (when $\maxnorm{\mathbf{a}} > 1$): \newline $\ell_2 = \left(\maxnorm{\mathbf{a}} \lim_{k \to \infty}  (1 + \maxnorm{\mathbf{a}}^{-k})^{1/k} \right)^{-1} = \maxnorm{\mathbf{a}}^{-1} > 0$.
\end{itemize}
Observe that in both subcases, it holds that $\ell_2 = \min \{ \maxnorm{\mathbf{a}}^{-1},1 \} = R_w$.

\item \emph{Case 2} (when $\maxnorm{\mathbf{a}} = 1$): Let $C \coloneqq 1$ and $h_2(k) \coloneqq \maxnorm{\mathbf{b}} k + 1$. Then, \linebreak $\ell_2 \coloneqq \left(\limsup_{k \to \infty}  \left(h_2(k)\right)^{1/k} \right)^{-1} = \left(\lim_{k \to \infty}  e^{\log(\maxnorm{\mathbf{b}} k + 1)/k} \right)^{-1} = 1 > 0$. Note that $\ell_2 = \min \{ \maxnorm{\mathbf{a}}^{-1},1 \} = R_w$ as well.
\end{itemize}
By exploiting Proposition~\ref{prop:Sufficient_condition_for_unconditional_uniform_and_pointwise_convergence_of_multiple_power_series} once again we obtain, in both cases, the desired result for the power series $\sum_{k=0}^{\infty} \sum_{n=0}^{\infty}  t_k(n) z^n w^k$.

\section{Proof of Theorem~\ref{thrm:Theorem_1}} \label{sec:Proof_Theorem_1}

\subsection{Proof of Statement~\ref{statement:1.a}}

Given a $z \in \DD_{*} \coloneqq \DD \setminus\{0\}$, we should consider two cases; recall that conditions \ref{cond:1.a} and \ref{cond:1.b} hold.
\begin{itemize}
	\item \emph{Case 1} (when $\mu_r - \lambda_r \geq 0$): The numerator of the left-hand-side expression in \eqref{eq:Partial_fraction_decomposition_wrt_u} equals $1$ (so it has zero degree) and is coprime with the polynomial denominator $u^{\mu_r - \lambda_r + 1} (u^{\lambda_r}-z^m)$ that has degree equal to $\mu_r + 1 \geq 1 > 0$.   

	\item \emph{Case 2} (when $\mu_r - \lambda_r < 0 \iff \mu_r - \lambda_r \leq -1$): The polynomial numerator now becomes $u^{-(\mu_r - \lambda_r + 1)}$ (so its degree equals $\lambda_r - \mu_r - 1 \geq 0$) and is coprime with the polynomial denominator $u^{\lambda_r}-z^m$, because $z \neq 0$, whose degree equals $\lambda_r > \lambda_r - \mu_r - 1$ (since $\mu_r \geq 0 > -1$).
\end{itemize}
In both cases, Lemma~\ref{lem:Partial_fraction_decomposition} is applicable and the statement follows immediately.

\subsection{Proof of Statement~\ref{statement:1.b}}

Fix a $z \in \DD_{*} \coloneqq \DD \setminus\{0\}$ and a $k \in \NN$. Let us choose a radius $\rho \in \RR_{+}$ so that it satisfies the condition: $\left( \abs{z}^{a_r^{-1}} = \right) \abs{z}^{m/{\lambda_r}} < \rho < 1$ for all $r \in \mathcal{R}(m)$, or equivalently $\max\left\{ \abs{z}^{a_r^{-1}} : r \in \mathcal{R}(m) \right\} < \rho < 1$. Since $0 < \abs{z} < 1$ and the function $y(x) = c^x$ is decreasing in $\RR$ for any constant $0 < c < 1$, the condition ultimately becomes $\abs{z}^{\maxnorm{\mathbf{a}}^{-1}} < \rho < 1$. In fact, $\rho$ is appropriately chosen to justify the interchange of summation and contour-integration; see below for more details. 

Based on \eqref{eq:Generating_function_f_k} and using Lemmas~\ref{lem:Iverson-bracket_sum_identities},~\ref{lem:Integral_representation_Kronecker_delta}~and~\ref{lem:Cauchy_integral_formulas}, we obtain

\begin{align*} 
f_k(z) & \stackrel{\mathclap{\eqref{eq:Generalized_Collatz_dynamics}}}{=} \sum_{n=0}^{\infty} t_{k-1}(t(n)) z^n \stackrel{\mathclap{\eqref{eq:Sum_q_r_Iverson_bracket_identity}}}{=} \sum_{n=0}^{\infty} \sum_{q=0}^{\infty} \sum_{r=0}^{m-1} t_{k-1}(t(n)) z^n \iverson{n = q m + r} \\
&  = \sum_{r=0}^{m-1} \sum_{q=0}^{\infty} t_{k-1}( q \lambda_r + \mu_r) z^{q m + r} \sum_{n=0}^{\infty} \iverson{n = q m + r}   \\
&  \stackrel{\mathclap{\eqref{eq:Sum_n_Iverson_bracket_identity}}}{=}  \sum_{r=0}^{m-1} \sum_{q=0}^{\infty} t_{k-1}( q \lambda_r + \mu_r ) z^{q m + r}    \\
& =  \sum_{r=0}^{m-1} \sum_{q=0}^{\infty} \sum_{n=0}^{\infty} \iverson{n = q \lambda_r + \mu_r} t_{k-1}(n) z^{q m + r}      \\
& \stackrel{\mathclap{\eqref{eq:Kronecker_delta_integral_representation}}}{=} \sum_{r=0}^{m-1} \sum_{q=0}^{\infty} \sum_{n=0}^{\infty} \left( \frac{1}{2\pi\iu} \cintpos_{\mathcal{C}(\rho)} {u^{n- q \lambda_r - \mu_r -1}} \rd u \right) t_{k-1}(n) z^{q m + r}   \\
& = \sum_{r=0}^{m-1} \frac{z^r}{2\pi\iu} \cintpos_{\mathcal{C}(\rho)} \frac{1}{u^{\mu_r + 1}} \left( \sum_{q=0}^{\infty} \left(\frac{z^m}{u^{\lambda_r}}\right)^q \right) \left( \sum_{n=0}^{\infty} t_{k-1}(n) u^n \right) \! \rd u        \\ 
& = \sum_{r=0}^{m-1} \frac{z^r}{2\pi\iu} \cintpos_{\mathcal{C}(\rho)} \frac{1}{u^{\mu_r + 1}} \frac{1}{1-\frac{z^m}{u^{\lambda_r}}} f_{k-1}(u)   \rd u      \\
& =  \sum_{r=0}^{m-1} \frac{z^r}{2\pi\iu} \cintpos_{\mathcal{C}(\rho)} \frac{1}{u^{\mu_r - \lambda_r + 1} (u^{\lambda_r}-z^m)} f_{k-1}(u)   \rd u      \\
& \stackrel{\mathclap{\eqref{eq:Partial_fraction_decomposition_wrt_u}}}{=} \sum_{r=0}^{m-1} z^r \left( \sum_{l=0}^{\mu_r - \lambda_r} {\frac{\sigma_{r,l}(z)}{2\pi\iu} \cintpos_{\mathcal{C}(\rho)} \frac{f_{k-1}(u)}{u^{l+1}} \rd u} + \sum_{l=0}^{\lambda_r-1} {\frac{\tau_{r,l}(z)}{2\pi\iu} \cintpos_{\mathcal{C}(\rho)} \frac{f_{k-1}(u)}{u - \phi_{r,l}(z)} \rd u} \right)     \\
&  \stackrel{\mathclap{\eqref{eq:Coefficients_a_n},\eqref{eq:Residue_with_independent_poles}}}{=}  \quad\!  \sum_{r=0}^{m-1} z^r \left( \sum_{l=0}^{\mu_r - \lambda_r} { \frac{\sigma_{r,l}(z)}{l!} f_{k-1}^{(l)}(0) } + \sum_{l=0}^{\lambda_r-1} {{\tau_{r,l}(z)} f_{k-1}(\phi_{r,l}(z))} \right)    ,
\end{align*}
with initial condition $f_0(z) = \sum_{n=0}^{\infty} {n z^n} = \frac{z}{(1-z)^2}$.

Note that the interchange of the order of summations in the third equality is valid, because by Proposition~\ref{prop:Unconditional_uniform_and_pointwise_convergence_of_GFs} the power series $f_k(z) = \sum_{n=0}^{\infty} t_k(n) z^n$ is  unconditionally pointwise convergent on $\DD$ (and therefore on $\DD_{*} \subseteq \DD$). 

In addition, according to \cite[Theorem~7.16 and its Corollary]{Rudin_1976}, the interchange of the order of summation and contour-integration in the seventh equality is possible due to the particular choice of the radius $\rho$. More precisely, this choice ensures \emph{unconditional uniform convergence on the integration contour $\mathcal{C}(\rho)$} of the (single) infinite series: $\sum_{q=0}^{\infty} \left(\frac{z^m}{u^{\lambda_r}}\right)^q$ (since $\abs{{z^m}/{u^{\lambda_r}}} = {\abs{z}^m}/{\rho^{\lambda_r}} < 1$, for all $u \in \mathcal{C}(\rho)$) and $\sum_{n=0}^{\infty} t_{k-1}(n) u^n$ (since $\abs{u} = \rho < 1$, for all $u \in \mathcal{C}(\rho)$); cf.~Proposition~\ref{prop:Sufficient_condition_for_unconditional_uniform_and_pointwise_convergence_of_multiple_power_series}~and~Proposition~\ref{prop:Unconditional_uniform_and_pointwise_convergence_of_GFs}, respectively.

\subsection{Proof of Statement~\ref{statement:1.c}}

According to \eqref{eq:Generating_function_F} and using Lemma~\ref{lem:Cauchy_integral_formulas}, we get 
\begin{align*} 
F(z,w) & = \sum_{k=0}^{\infty} f_k(z) w^k = f_0(z) + \sum_{k=1}^{\infty} f_k(z) w^k   \\
& \stackrel{\mathclap{\eqref{eq:Functional_recurrence_f_k}}}{=} f_0(z) +  \sum_{r=0}^{m-1} z^r \left( \sum_{l=0}^{\mu_r - \lambda_r} {{\frac{\sigma_{r,l}(z)}{l!} \sum_{k=1}^{\infty} f_{k-1}^{(l)}(0) w^k}} + \sum_{l=0}^{\lambda_r-1} {{\tau_{r,l}(z)} \sum_{k=1}^{\infty} f_{k-1}(\phi_{r,l}(z)) w^k} \right)     \\
& = f_0(z) + w \sum_{r=0}^{m-1} z^r \left( \sum_{l=0}^{\mu_r - \lambda_r} {{\frac{\sigma_{r,l}(z)}{l!} \sum_{k=0}^{\infty} f_{k}^{(l)}(0) w^k}} + \sum_{l=0}^{\lambda_r-1} {{\tau_{r,l}(z)} \sum_{k=0}^{\infty} f_{k}(\phi_{r,l}(z)) w^k} \right)    \\
& \stackrel{\mathclap{\eqref{eq:Uniform_convergence_of_series_of_derivatives}}}{=} f_0(z) + w \sum_{r=0}^{m-1} z^r \left( \sum_{l=0}^{\mu_r - \lambda_r} {{\frac{\sigma_{r,l}(z)}{l!} F_{*}^{(l)}(0,w)}} + \sum_{l=0}^{\lambda_r-1} {{\tau_{r,l}(z)} F(\phi_{r,l}(z),w)} \right)    ,
\end{align*} 
where in the last equality we also used Proposition~\ref{prop:Unconditional_uniform_convergence_implies_unconditional_pointwise_convergence}. 

This concludes the proof of Theorem~\ref{thrm:Theorem_1}.

\section{Proof of Proposition~\ref{prop:Unconditional_uniform_and_pointwise_convergence_of_vector_GFs}} \label{sec:Proof_Proposition_2}

\begin{lemma}[Rough upper bound on $\maxnorm{\mathbf{t}_k(\mathbf{n})}$] \label{lem:Rough_upper_bound_on_t_k_multi-dimensional}
It holds that  
\begin{equation*}
\maxnorm{\mathbf{t}_k(\mathbf{n})} \leq \begin{cases}
(d \maxnorm{\mathbf{A}_{\circ}})^k (\maxnorm{\mathbf{n}} + \chi) - \chi , & \mathrm{if}\ d \maxnorm{\mathbf{A}_{\circ}} \neq 1  , \\
\maxnorm{\mathbf{n}} + \maxnorm{\mathbf{B}} k , & \mathrm{if}\ d \maxnorm{\mathbf{A}_{\circ}} = 1  ,
\end{cases}
\end{equation*}
for all $(k,\mathbf{n}) \in \NN_0^{d+1}$, where $\chi \coloneqq \frac{\maxnorm{\mathbf{B}}}{d \maxnorm{\mathbf{A}_{\circ}} - 1}$. 
\end{lemma}

\begin{proof}
Based on \eqref{eq:Multi-dimensional_generalized_Collatz_dynamics} and \eqref{eq:Multi-dimensional_generalized_Collatz_map}, we have
\begin{equation*} 
\begin{split}
\maxnorm{\mathbf{t}_k(\mathbf{n})} = \maxnorm{\mathbf{t}(\mathbf{t}_{k-1}(\mathbf{n}))} & \leq \maxnorm{\mathbf{A}_{\mathbf{r}} \mathbf{t}_{k-1}(\mathbf{n})} + \maxnorm{\mathbf{b}_{\mathbf{r}}}  \\
& \leq d \maxnorm{\mathbf{A}_{\mathbf{r}}}\maxnorm{\mathbf{t}_{k-1}(\mathbf{n})} + \maxnorm{\mathbf{b}_{\mathbf{r}}},  
\end{split}
\end{equation*}
if $\mathbf{r} \in {\mathcal{R}(\mathbf{m})}$ and $\mathbf{t}_{k-1}(\mathbf{n}) \equiv \mathbf{r} \pmod{\mathbf{m}}$. The first inequality follows form the triangle inequality for norms, while the second one from the fact that: for any $k \times m$ matrix $\mathbf{D}$ and any $m \times n$ matrix $\mathbf{E}$,  it holds that $\maxnorm{\mathbf{D} \mathbf{E}} \leq m \maxnorm{\mathbf{D}} \maxnorm{\mathbf{E}}$. Furthermore, $\maxnorm{\mathbf{A}_{\mathbf{r}}} \leq \maxnorm{\mathbf{A}_{\circ}}$ and $\maxnorm{\mathbf{b}_{\mathbf{r}}} \leq \maxnorm{\mathbf{B}}$ for all $\mathbf{r} \in {\mathcal{R}(\mathbf{m})}$, thus yielding  
\begin{equation*}
\maxnorm{\mathbf{t}_k(\mathbf{n})} \leq (d \maxnorm{\mathbf{A}_{\circ}}) \maxnorm{\mathbf{t}_{k-1}(\mathbf{n})} + \maxnorm{\mathbf{B}}  ,
\end{equation*}
for all $k \in \NN_0$ and $\mathbf{n} \in \NN_0^d$.

Similar to the inductive proof of Lemma~\ref{lem:Rough_upper_bound_on_t_k}, by replacing $\maxnorm{\mathbf{a}} \mapsto d \maxnorm{\mathbf{A}_{\circ}}$, $\maxnorm{\mathbf{b}} \mapsto \maxnorm{\mathbf{B}}$ therein and using the initial condition $\mathbf{t}_0(\mathbf{n}) = \mathbf{n}$ (cf. Eqs.~\eqref{eq:Triangle_inequality_for_t_k} and~\eqref{eq:Inequality_by_induction_on_k}), we obtain 
\begin{align*} 
\maxnorm{\mathbf{t}_k(\mathbf{n})} & \leq (d \maxnorm{\mathbf{A}_{\circ}})^k \maxnorm{\mathbf{n}} + \maxnorm{\mathbf{B}} \sum_{j=0}^{k-1} {(d \maxnorm{\mathbf{A}_{\circ}})^j}   \\
& = \begin{cases} 
(d \maxnorm{\mathbf{A}_{\circ}})^k \maxnorm{\mathbf{n}} + \maxnorm{\mathbf{B}} \frac{(d \maxnorm{\mathbf{A}_{\circ}})^k - 1}{d \maxnorm{\mathbf{A}_{\circ}} - 1} , & \mathrm{if}\ d \maxnorm{\mathbf{A}_{\circ}} \neq 1  , \\
(d \maxnorm{\mathbf{A}_{\circ}})^k \maxnorm{\mathbf{n}} + \maxnorm{\mathbf{B}} k , & \mathrm{if}\ d \maxnorm{\mathbf{A}_{\circ}} = 1  .    
\end{cases}     \\
& = \begin{cases}
(d \maxnorm{\mathbf{A}_{\circ}})^k (\maxnorm{\mathbf{n}} + \chi) - \chi , & \mathrm{if}\ d \maxnorm{\mathbf{A}_{\circ}} \neq 1  , \\
\maxnorm{\mathbf{n}} + \maxnorm{\mathbf{B}} k , & \mathrm{if}\ d \maxnorm{\mathbf{A}_{\circ}} = 1  .
\end{cases}   
\end{align*} 
\end{proof}

By virtue of Lemma~\ref{lem:Rough_upper_bound_on_t_k_multi-dimensional} and the triangle inequality, we obtain  
\begin{align*} 
\maxnorm{\mathbf{t}_k(\mathbf{n})} & \leq \begin{cases}
\abs{(d \maxnorm{\mathbf{A}_{\circ}})^k \maxnorm{\mathbf{n}} + (d \maxnorm{\mathbf{A}_{\circ}})^k \chi - \chi} , & \mathrm{if}\ d \maxnorm{\mathbf{A}_{\circ}} \neq 1  , \\
\maxnorm{\mathbf{n}} + \maxnorm{\mathbf{B}} k , & \mathrm{if}\ d \maxnorm{\mathbf{A}_{\circ}} = 1  ,
\end{cases}   \\
& \leq \begin{cases} 
(d \maxnorm{\mathbf{A}_{\circ}})^k \maxnorm{\mathbf{n}} + (d \maxnorm{\mathbf{A}_{\circ}})^k \abs{\chi} + \abs{\chi} , & \mathrm{if}\ d \maxnorm{\mathbf{A}_{\circ}} \neq 1  , \\
\maxnorm{\mathbf{n}} + \maxnorm{\mathbf{B}} k , & \mathrm{if}\ d \maxnorm{\mathbf{A}_{\circ}} = 1  ,
\end{cases}   \\
& \leq \begin{cases} 
(\abs{\chi} + 1) ((d \maxnorm{\mathbf{A}_{\circ}})^k + 1) (\maxnorm{\mathbf{n}} + 1) , & \mathrm{if}\ d \maxnorm{\mathbf{A}_{\circ}} \neq 1  , \\
(\maxnorm{\mathbf{B}} k + 1) (\maxnorm{\mathbf{n}} + 1) , & \mathrm{if}\ d \maxnorm{\mathbf{A}_{\circ}} = 1  ,
\end{cases}   \\
& \leq \begin{cases} 
(\abs{\chi} + 1) ((d \maxnorm{\mathbf{A}_{\circ}})^k + 1) {\prod_{j=1}^d (n_j + 1)} , & \mathrm{if}\ d \maxnorm{\mathbf{A}_{\circ}} \neq 1  , \\
(\maxnorm{\mathbf{B}} k + 1) {\prod_{j=1}^d (n_j + 1)} , & \mathrm{if}\ d \maxnorm{\mathbf{A}_{\circ}} = 1  ,
\end{cases}  
\end{align*} 
where the last inequality follows from the relation: $\maxnorm{\mathbf{n}} + 1 \leq \sum_{j=1}^d {n_j} + 1 \leq \prod_{j=1}^d (n_j + 1)$, for all $\mathbf{n} \in \NN_0^d$. Subsequently, let us choose $h_j(n) = h(n) \coloneqq n+1$, for all $j \in \{1,\dots,d\}$, and $h_{d+1} (k) \coloneqq (d \maxnorm{\mathbf{A}_{\circ}})^k + 1$ when $d \maxnorm{\mathbf{A}_{\circ}} \neq 1$, or $h_{d+1} (k) \coloneqq \maxnorm{\mathbf{B}} k + 1$ when $d \maxnorm{\mathbf{A}_{\circ}} = 1$. As in the proof of Proposition~\ref{prop:Unconditional_uniform_and_pointwise_convergence_of_GFs} (one-dimensional case) with the substitutions $\maxnorm{\mathbf{a}} \mapsto d \maxnorm{\mathbf{A}_{\circ}}$ and $\maxnorm{\mathbf{b}} \mapsto \maxnorm{\mathbf{B}}$, we get $\ell_j \coloneqq \left(\limsup_{n \to \infty}  \left(h_j(n)\right)^{1/n} \right)^{-1} = 1 > 0$, for all $j \in \{1,\dots,d\}$, and $\ell_{d+1} \coloneqq  \left(\limsup_{k \to \infty}  \left(h_{d+1}(k)\right)^{1/k} \right)^{-1}  = \min \{ (d \maxnorm{\mathbf{A}_{\circ}})^{-1},1 \} = R_w > 0$. If we consider each component power series $\sum_{\mathbf{n} \in \NN_0^d} {{t_{j,k}(\mathbf{n})} \, \mathbf{z}^{\mathbf{n}}}$ and $\sum_{k \in \NN_0} \sum_{\mathbf{n} \in \NN_0^d} {t_{j,k}(\mathbf{n}) \, \mathbf{z}^{\mathbf{n}} w^k}$ with $\abs{t_{j,k}(\mathbf{n})} \leq \maxnorm{\mathbf{t}_k(\mathbf{n})}$, then the desired result follows directly from Proposition~\ref{prop:Sufficient_condition_for_unconditional_uniform_and_pointwise_convergence_of_multiple_power_series}. Note that $\maxnorm{\mathbf{z}} \leq \rho_{\mathbf{z}}$ is equivalent to $\abs{z_j} \leq \rho_{\mathbf{z}}$ for all $j \in \{1,\dots,d\}$.

\section{Proof of Theorem~\ref{thrm:Theorem_2}} \label{sec:Proof_Theorem_2}

\subsection{Proof of Statement~\ref{statement:2.a}} \label{sec:Proof_Statement_2.a}

Let $\textbf{z} \in \DD_{*}^d \coloneqq (\DD \setminus\{0\})^d$, which implies that $z_j \neq 0$ for all $j \in \{1,\dots,d\}$. The finite set $\mathcal{L}(\mathbf{r}) \subseteq \NN_0^{d+1}$ contains multi-index-plus-index pairs $(\boldsymbol{\ell},\nu)$ and depends on $\mathbf{r}$. Since a given multi-index $\boldsymbol{\ell}$, e.g., $\boldsymbol{\ell} = \mathbf{0}$, may occur more than once in the final summation of \eqref{eq:Partial_fraction_decomposition_wrt_u_multi-dimensional}, we need the auxiliary index $\nu$ to distinguish its terms. Without loss of generality, the index $\nu$ ranges from $0$ to $(\operatorname{multiplicity}(\boldsymbol{\ell})-1)$, where $\operatorname{multiplicity}(\boldsymbol{\ell})$ is the number of occurrences of the multi-exponent $\boldsymbol{\ell}+\mathbf{1}$ in the final summation of \eqref{eq:Partial_fraction_decomposition_wrt_u_multi-dimensional}. In addition, observe that $(\mathbf{0},0) \in \mathcal{L}(\mathbf{r})$ because of condition \ref{cond:2.a}, $\eta_{\mathbf{r},\boldsymbol{\ell},\nu}(\mathbf{z}) = \prod_{j=1}^d {\overline{\eta}_{\mathbf{r},\boldsymbol{\ell},\nu,j}(z_j)}$ with $\overline{\eta}_{\mathbf{r},\boldsymbol{\ell},\nu,j}(z_j) \in \bigcup_l \{\sigma_{\mathbf{r},j,l}(z_j)\} \cup \bigcup_l\{\tau_{\mathbf{r},j,l}(z_j)\}$, and $\boldsymbol{\varphi}_{\mathbf{r},\boldsymbol{\ell},\nu}(\mathbf{z}) = [\varphi_{\mathbf{r},\boldsymbol{\ell},\nu,1}(z_1),\dots,\varphi_{\mathbf{r},\boldsymbol{\ell},\nu,d}(z_d)]^{\top}$ with $\varphi_{\mathbf{r},\boldsymbol{\ell},\nu,j}(z_j) \in \{0\} \cup \bigcup_l \{\phi_{\mathbf{r},j,l}(z_j)\}$. 

More precisely, we will give a proof by induction on $d$, which can also be used as a \emph{recursive/iterative method} to compute the $\mathcal{L}(\mathbf{r})$, $\{\eta_{\mathbf{r},\boldsymbol{\ell},\nu}(\mathbf{z})\}$ and $\{\boldsymbol{\varphi}_{\mathbf{r},\boldsymbol{\ell},\nu}(\mathbf{z})\}$. Recall that the uniqueness of $\{\sigma_{\mathbf{r},j,l}(z_j)\}$ and $\{\tau_{\mathbf{r},j,l}(z_j)\}$ is guaranteed by \eqref{eq:Partial_fraction_decomposition_wrt_u}.

\emph{Basis step}: For $d=1$, given $\{\sigma_{r,l}(z)\}$ and $\{\tau_{r,l}(z)\}$, we can write 
\begin{equation*}
\frac{1}{u^{\mu_r - \lambda_r + 1} (u^{\lambda_r}-z^m)} = \sum_{l=0}^{\mu_r - \lambda_r} {\frac{\sigma_{r,l}(z)}{u^{l+1}}} + \sum_{l=0}^{\lambda_r-1} {\frac{\tau_{r,l}(z)}{u - \phi_{r,l}(z)}} = \sum_{(\ell,\nu) \in \mathcal{L}(r)} \frac{\eta_{r,\ell,\nu}(z)}{\left(u - \varphi_{r,\ell,\nu}(z)\right)^{\ell+1}}   ,
\end{equation*}
where $\mathcal{L}(r) =  \bigcup_{\ell = 1}^{\mu_r - \lambda_r} \{(\ell,0)\}  \cup  \bigcup_{\nu = 0}^{\lambda_r - \iverson{\mu_r < \lambda_r}} \{(0,\nu)\}  \subseteq \NN_0^2$, $\eta_{r,\ell,0}(z) = \sigma_{r,\ell}(z)$ and $\varphi_{r,\ell,0}(z) = 0$ for all $\ell \in \{1,\dots,\mu_r - \lambda_r\}$, $\eta_{r,0,\nu}(z) = \tau_{r,\nu}(z)$ and $\varphi_{r,0,\nu}(z) = \phi_{r,\nu}(z)$ for all $\nu \in \{0,\dots,\lambda_r-1\}$, $\eta_{r,0,\lambda_r}(z) = \sigma_{r,0}(z)$ and $\varphi_{r,0,\lambda_r}(z) = 0$.

\emph{Inductive step}: Let us choose an arbitrary integer $d \geq 2$. Suppose that the partial fraction decomposition in \eqref{eq:Partial_fraction_decomposition_wrt_u_multi-dimensional} is true for dimension $d-1$ (\emph{inductive hypothesis}). Then, we can prove that it is also true for dimension $d$. In particular, by using the identity $\prod_{j=1}^d {(\cdot)_j} = \left( \prod_{j=1}^{d-1} {(\cdot)_j} \right) {(\cdot)_d}$ and the basis step, we obtain
\begin{align*} 
\prod_{j=1}^d {\frac{1}{u_j^{\mu_{\mathbf{r},j} - \lambda_{\mathbf{r},j} + 1} (u_j^{\lambda_{\mathbf{r},j}}-z_j^{m_j})}} & = \left( \prod_{j=1}^{d-1} {\frac{1}{u_j^{\mu_{\mathbf{r},j} - \lambda_{\mathbf{r},j} + 1} (u_j^{\lambda_{\mathbf{r},j}}-z_j^{m_j})}} \right) {\frac{1}{u_d^{\mu_{\mathbf{r},d} - \lambda_{\mathbf{r},d} + 1} (u_d^{\lambda_{\mathbf{r},d}}-z_d^{m_d})}}  \\
& = \left( \sum_{(\boldsymbol{\ell}',\nu') \in \mathcal{L}'(\mathbf{r})} {\frac{\eta'_{\mathbf{r},\boldsymbol{\ell}',\nu'}(\mathbf{z}_{-d})}{\left( \mathbf{u}_{-d} - \boldsymbol{\varphi}'_{\mathbf{r},\boldsymbol{\ell}',\nu'}(\mathbf{z}_{-d}) \right)^{\boldsymbol{\ell}'+\mathbf{1}}}} \right)  \\
& \quad\; . \left(  \sum_{(\ell'',\nu'') \in \mathcal{L}''(\mathbf{r})} {\frac{\eta''_{\mathbf{r},\ell'',\nu''}(z_d)}{\left( u_d - \varphi''_{\mathbf{r},\ell'',\nu''}(z_d) \right)^{\ell''+1}}} \right)     \\
& =  \sum_{\substack{(\boldsymbol{\ell}',\nu') \in \mathcal{L}'(\mathbf{r}), \\ (\ell'',\nu'') \in \mathcal{L}''(\mathbf{r})}}  \frac{\eta'_{\mathbf{r},\boldsymbol{\ell}',\nu'}(\mathbf{z}_{-d}) \, \eta''_{\mathbf{r},\ell'',\nu''}(z_d)}{\left( \mathbf{u}_{-d} - \boldsymbol{\varphi}'_{\mathbf{r},\boldsymbol{\ell}',\nu'}(\mathbf{z}_{-d}) \right)^{\boldsymbol{\ell}'+\mathbf{1}} \left( u_d - \varphi''_{\mathbf{r},\ell'',\nu''}(z_d) \right)^{\ell''+1}}         \\
& =  \sum_{(\boldsymbol{\ell},\nu) \in \mathcal{L}(\mathbf{r})} {\frac{\eta_{\mathbf{r},\boldsymbol{\ell},\nu}(\mathbf{z})}{\left( \mathbf{u} - \boldsymbol{\varphi}_{\mathbf{r},\boldsymbol{\ell},\nu}(\mathbf{z}) \right)^{\boldsymbol{\ell}+\mathbf{1}}}}     ,    
\end{align*} 
where $\mathcal{L}'(\mathbf{r}) \subseteq \NN_0^d$, $\boldsymbol{\varphi}'_{\mathbf{r},\boldsymbol{\ell}',\nu'}(\mathbf{z}_{-d}) \in \CC^{d-1}$, $\mathcal{L}''(\mathbf{r}) \subseteq \NN_0^2$, $\mathcal{L}(\mathbf{r}) \subseteq \NN_0^{d+1}$, $\boldsymbol{\varphi}_{\mathbf{r},\boldsymbol{\ell},\nu}(\mathbf{z}) \in \CC^d$. Specifically, given the sets $\mathcal{L}'(\mathbf{r})$ and $\mathcal{L}''(\mathbf{r})$, the set $\mathcal{L}(\mathbf{r})$ is constructed as follows. First, for each $\boldsymbol{\ell} \in \NN_0^d$, let us define the 
\begin{equation*}
\operatorname{multiplicity}(\boldsymbol{\ell}) \coloneqq \card{\{(\boldsymbol{\ell}',\nu',\ell'',\nu'') \in \mathcal{L}'(\mathbf{r}) \times \mathcal{L}''(\mathbf{r}) : \boldsymbol{\ell} = (\boldsymbol{\ell}',\ell'')\}} .
\end{equation*}
Then, for every $\boldsymbol{\ell} \in \NN_0^d$ with $\operatorname{multiplicity}(\boldsymbol{\ell}) \geq 1$,\footnote{There exist \emph{finitely many} such $\boldsymbol{\ell}$, because the sets $\mathcal{L}'(\mathbf{r})$ and $\mathcal{L}''(\mathbf{r})$ are finite.} we define the partial function $h_{\boldsymbol{\ell}} \colon \NN_0^2 \to \{0,\dots,\operatorname{multiplicity}(\boldsymbol{\ell})-1\}$ so that $h_{\boldsymbol{\ell}}(\nu',\nu'') = \nu$ if the position of the pair $(\nu',\nu'')$ in the set 
\begin{equation*}
\mathcal{H}({\boldsymbol{\ell}}) \coloneqq \{(\nu_1,\nu_2) \in \NN_0^2: \boldsymbol{\ell} = (\boldsymbol{\ell}',\ell''), \,  (\boldsymbol{\ell}',\nu_1,\ell'',\nu_2) \in \mathcal{L}'(\mathbf{r}) \times \mathcal{L}''(\mathbf{r})\} 
\end{equation*}
is equal to $\nu$ (assuming lexicographic order and starting from position $0$). Note that $h_{\boldsymbol{\ell}}$ is a bijective function on its domain of definition, so its inverse $h_{\boldsymbol{\ell}}^{-1}$ exists. Now, we can define the set 
\begin{equation*}
\begin{split}
\mathcal{L}(\mathbf{r}) \coloneqq \{(\boldsymbol{\ell},\nu) \in \NN_0^{d+1}: & \, \exists (\boldsymbol{\ell}',\nu',\ell'',\nu'') \in \mathcal{L}'(\mathbf{r}) \times \mathcal{L}''(\mathbf{r}) \ \text{such that} \\
& \ \boldsymbol{\ell} = (\boldsymbol{\ell}',\ell'') \ \text{and} \ \nu = h_{\boldsymbol{\ell}}(\nu',\nu'')\}  .
\end{split}
\end{equation*}
Finally, $\eta_{\mathbf{r},\boldsymbol{\ell},\nu}(\mathbf{z}) \coloneqq \eta'_{\mathbf{r},\boldsymbol{\ell}',\nu'}(\mathbf{z}_{-d}) \, \eta''_{\mathbf{r},\ell'',\nu''}(z_d)$ and $\boldsymbol{\varphi}_{\mathbf{r},\boldsymbol{\ell},\nu}(\mathbf{z}) \coloneqq [\boldsymbol{\varphi}'_{\mathbf{r},\boldsymbol{\ell}',\nu'}(\mathbf{z}_{-d})^{\top},\varphi''_{\mathbf{r},\ell'',\nu''}(z_d)]^{\top}$, where $(\boldsymbol{\ell}',\ell'') = \boldsymbol{\ell}$ and $(\nu',\nu'') = h_{\boldsymbol{\ell}}^{-1}(\nu)$. 

Hence, the multi-dimensional partial fraction decomposition in \eqref{eq:Partial_fraction_decomposition_wrt_u_multi-dimensional} holds for all $d \in \NN$.

\subsection{Proof of Statement~\ref{statement:2.b}}

Fix a $\mathbf{z} \in \DD_{*}^d \coloneqq (\DD \setminus\{0\})^d$ and a $k \in \NN$. Let us choose a poly-radius $\boldsymbol{\rho} \in \RR_{+}^d$ so that it satisfies the condition: $\abs{z_j}^{m_j/\lambda_{\mathbf{r},j}} < \rho_j < 1$ for all $\mathbf{r} \in \mathcal{R}(\mathbf{m})$ and $j \in \{1,\dots,d\}$, or equivalently $\max\left\{ \abs{z_j}^{m_j/\lambda_{\mathbf{r},j}} : \mathbf{r} \in \mathcal{R}(\mathbf{m}) \right\} < \rho_j < 1$ for all $j \in \{1,\dots,d\}$. Since $0 < \abs{z_j} < 1$ and the function $y(x) = c^x$ is decreasing in $\RR$ for any constant $0 < c < 1$, the condition ultimately becomes $\abs{z_j}^{m_j/\lambda_j^{\max}} < \rho_j < 1$, where $\lambda_j^{\max} \coloneqq \max\{ \lambda_{\mathbf{r},j} : \mathbf{r} \in \mathcal{R}(\mathbf{m}) \}$, for all $j \in \{1,\dots,d\}$. 
In essence, $\boldsymbol{\rho}$ is appropriately chosen to ensure the interchange of summation and contour-integration; see below for more details. 

Based on \eqref{eq:Vector_generating_function_f_k} and using Lemmas~\ref{lem:Multi-dimensional_Iverson-bracket_sum_identities},~\ref{lem:Integral_representation_multi-dimensional_Kronecker_delta}~and~\ref{lem:Cauchy_integral_formulas}, we obtain 

\begin{align*} 
\mathbf{f}_k(\mathbf{z}) &  \stackrel{\mathclap{\eqref{eq:Multi-dimensional_generalized_Collatz_dynamics}}}{=}  \sum_{\mathbf{n} \in \NN_0^d} {\mathbf{t}_{k-1}(\mathbf{t}(\mathbf{n})) \, \mathbf{z}^{\mathbf{n}}}  \stackrel{\mathclap{\eqref{eq:Sum_q_r_Iverson_bracket_identity_multi-dimensional}}}{=}  \sum_{\mathbf{n} \in \NN_0^d}  \sum_{\mathbf{q} \in \NN_0^d} \sum_{\mathbf{r} \in {\mathcal{R}(\mathbf{m})}} {\mathbf{t}_{k-1}(\mathbf{t}(\mathbf{n})) \, \mathbf{z}^{\mathbf{n}} \iverson{\mathbf{n} = \mathbf{q} \odot \mathbf{m} + \mathbf{r}}}   \\
& = \sum_{\mathbf{r} \in {\mathcal{R}(\mathbf{m})}} \sum_{\mathbf{q} \in \NN_0^d} {\mathbf{t}_{k-1}( \mathbf{q} \odot \boldsymbol{\lambda}_{\mathbf{r}} + \boldsymbol{\mu}_{\mathbf{r}} ) \, \mathbf{z}^{\mathbf{q} \odot \mathbf{m} + \mathbf{r}} \sum_{\mathbf{n} \in \NN_0^d} \iverson{\mathbf{n} = \mathbf{q} \odot \mathbf{m} + \mathbf{r}} }  \\
&  \stackrel{\mathclap{\eqref{eq:Sum_n_Iverson_bracket_identity_multi-dimensional}}}{=} \sum_{\mathbf{r} \in {\mathcal{R}(\mathbf{m})}} \sum_{\mathbf{q} \in \NN_0^d} {\mathbf{t}_{k-1}( \mathbf{q} \odot \boldsymbol{\lambda}_{\mathbf{r}} + \boldsymbol{\mu}_{\mathbf{r}} ) \, \mathbf{z}^{\mathbf{q} \odot \mathbf{m} + \mathbf{r}} }   \\
& = \sum_{\mathbf{r} \in {\mathcal{R}(\mathbf{m})}} \sum_{\mathbf{q} \in \NN_0^d} \sum_{\mathbf{n} \in \NN_0^d}  {\iverson{\mathbf{n}= \mathbf{q} \odot \boldsymbol{\lambda}_{\mathbf{r}} + \boldsymbol{\mu}_{\mathbf{r}}} \mathbf{t}_{k-1}( \mathbf{n} ) \, \mathbf{z}^{\mathbf{q} \odot \mathbf{m} + \mathbf{r}}}   \\
& \stackrel{\mathclap{\eqref{eq:Kronecker_delta_integral_representation_multi-dimensional}}}{=}  \sum_{\mathbf{r} \in {\mathcal{R}(\mathbf{m})}} \sum_{\mathbf{q} \in \NN_0^d} \sum_{\mathbf{n} \in \NN_0^d}  { \left( \frac{1}{(2\pi\iu)^d} \cintpos_{\mathcal{C}(\boldsymbol{\rho})} {\mathbf{u}^{\mathbf{n} - \mathbf{q} \odot \boldsymbol{\lambda}_{\mathbf{r}} - \boldsymbol{\mu}_{\mathbf{r}} - \mathbf{1}} \rd \mathbf{u}} \right) \mathbf{t}_{k-1}( \mathbf{n} ) \, \mathbf{z}^{\mathbf{q} \odot \mathbf{m} + \mathbf{r}}}     \\
& =  \sum_{\mathbf{r} \in {\mathcal{R}(\mathbf{m})}} \frac{\mathbf{z}^\mathbf{r}}{(2\pi\iu)^d} \cintpos_{\mathcal{C}(\boldsymbol{\rho})} \frac{1}{\mathbf{u}^{\boldsymbol{\mu}_{\mathbf{r}}+\mathbf{1}}}   \left( \sum_{\mathbf{q} \in \NN_0^d} \frac{\mathbf{z}^{\mathbf{q} \odot \mathbf{m}}}{\mathbf{u}^{\mathbf{q} \odot \boldsymbol{\lambda}_\mathbf{r}}} \right)   \left( \sum_{\mathbf{n} \in \NN_0^d} \mathbf{t}_{k-1}(\mathbf{n}) \, \mathbf{u}^{\mathbf{n}} \right)  \! \rd \mathbf{u}       \\
& = \sum_{\mathbf{r} \in {\mathcal{R}(\mathbf{m})}} \frac{\mathbf{z}^\mathbf{r}}{(2\pi\iu)^d} \cintpos_{\mathcal{C}(\boldsymbol{\rho})} \frac{1}{\mathbf{u}^{\boldsymbol{\mu}_{\mathbf{r}}+\mathbf{1}}}  \left( \prod_{j=1}^d {\left( \sum_{q \in \NN_0} \left(\frac{z_j^{m_j}}{u_j^{\lambda_{\mathbf{r},j}}}\right)^q \right)}  \right) \mathbf{f}_{k-1}(\mathbf{u})  \rd \mathbf{u}        \\
& = \sum_{\mathbf{r} \in {\mathcal{R}(\mathbf{m})}} \frac{\mathbf{z}^\mathbf{r}}{(2\pi\iu)^d} \cintpos_{\mathcal{C}(\boldsymbol{\rho})}  \left( \prod_{j=1}^d {\frac{1}{u_j^{\mu_{\mathbf{r},j} - \lambda_{\mathbf{r},j} + 1} (u_j^{\lambda_{\mathbf{r},j}}-z_j^{m_j})}} \right)  \mathbf{f}_{k-1}(\mathbf{u})  \rd \mathbf{u}   \\
& \stackrel{\mathclap{\eqref{eq:Partial_fraction_decomposition_wrt_u_multi-dimensional}}}{=}  \sum_{\mathbf{r} \in {\mathcal{R}(\mathbf{m})}} \frac{\mathbf{z}^\mathbf{r}}{(2\pi\iu)^d} \cintpos_{\mathcal{C}(\boldsymbol{\rho})}  \left( \sum_{(\boldsymbol{\ell},\nu) \in \mathcal{L}(\mathbf{r})} {\frac{\eta_{\mathbf{r},\boldsymbol{\ell},\nu}(\mathbf{z})}{\left( \mathbf{u} - \boldsymbol{\varphi}_{\mathbf{r},\boldsymbol{\ell},\nu}(\mathbf{z}) \right)^{\boldsymbol{\ell}+\mathbf{1}}}} \right)  \mathbf{f}_{k-1}(\mathbf{u})  \rd \mathbf{u}       \\
& =  \sum_{\mathbf{r} \in {\mathcal{R}(\mathbf{m})}} \mathbf{z}^\mathbf{r} \sum_{(\boldsymbol{\ell},\nu) \in \mathcal{L}(\mathbf{r})} \frac{\eta_{\mathbf{r},\boldsymbol{\ell},\nu}(\mathbf{z})}{(2\pi\iu)^d} \cintpos_{\mathcal{C}(\boldsymbol{\rho})}   {\frac{\mathbf{f}_{k-1}(\mathbf{u})}{\left( \mathbf{u} - \boldsymbol{\varphi}_{\mathbf{r},\boldsymbol{\ell},\nu}(\mathbf{z}) \right)^{\boldsymbol{\ell}+\mathbf{1}}}}    \rd \mathbf{u}       \\
& \stackrel{\mathclap{\eqref{eq:Residue_with_independent_poles}}}{=}  \sum_{\mathbf{r} \in {\mathcal{R}(\mathbf{m})}} \mathbf{z}^\mathbf{r} \sum_{(\boldsymbol{\ell},\nu) \in \mathcal{L}(\mathbf{r})}  \frac{\eta_{\mathbf{r},\boldsymbol{\ell},\nu}(\mathbf{z})}{\boldsymbol{\ell}!}  {\mathbf{f}_{k-1}^{(\boldsymbol{\ell})}(\boldsymbol{\varphi}_{\mathbf{r},\boldsymbol{\ell},\nu}(\mathbf{z}))}   ,   
\end{align*} 
with initial condition $\mathbf{f}_0(\mathbf{z}) = \sum_{\mathbf{n} \in \NN_0^d} {\mathbf{n} \, \mathbf{z}^{\mathbf{n}}}$, or equivalently $f_{j,0}(\textbf{z}) = \sum_{\mathbf{n} \in \NN_0^d} {n_j \, \mathbf{z}^{\mathbf{n}}} = \linebreak \left(\sum_{n \in \NN_0} {n z_j^n} \right) \prod_{l \neq j} \left( \sum_{n \in \NN_0} {z_l^n} \right) = \frac{z_j}{(1-z_j)^2 \prod_{l \neq j} {(1-z_l)}}$ for all $j \in \{1,\dots,d\}$. Here, we have used the fact that $\sum_{n \in \NN_0} {n y^n} = \frac{y}{(1-y)^2}$ and $\sum_{n \in \NN_0} {y^n} = \frac{1}{1-y}$ for all $y \in \DD$. 

Note that the interchange of the order of summations in the third equality is possible, because by Proposition~\ref{prop:Unconditional_uniform_and_pointwise_convergence_of_vector_GFs} the vector power series $\mathbf{f}_k(\mathbf{z}) = \sum_{\mathbf{n} \in \NN_0^d} {{\mathbf{t}_k(\mathbf{n})} \, \mathbf{z}^{\mathbf{n}}}$ is  unconditionally pointwise convergent on $\DD^d$ (and therefore on $\DD_{*}^d \subseteq \DD^d$). 

Furthermore, based on \cite[Theorem~7.16 and its Corollary]{Rudin_1976}, the interchange of the order of summation and contour-integration in the seventh equality is valid because of the particular choice of the poly-radius $\boldsymbol{\rho}$. More precisely, this choice guarantees \emph{unconditional uniform convergence on the integration contour $\mathcal{C}(\boldsymbol{\rho})$} of the (multiple) infinite series: $\sum_{\mathbf{q} \in \NN_0^d} \frac{\mathbf{z}^{\mathbf{q} \odot \mathbf{m}}}{\mathbf{u}^{\mathbf{q} \odot \boldsymbol{\lambda}_\mathbf{r}}} = \sum_{\mathbf{q} \in \NN_0^d} \prod_{j=1}^d { \left(\frac{z_j^{m_j}}{u_j^{\lambda_{\mathbf{r},j}}}\right)^{q_j} }$ (since $\abs{{z_j^{m_j}}/{u_j^{\lambda_{\mathbf{r},j}}}} = {\abs{z_j}^{m_j}}/{\rho_j^{\lambda_{\mathbf{r},j}}} < 1$, for all $j \in \{1,\dots,d\}$ and all $\mathbf{u} \in \mathcal{C}(\boldsymbol{\rho})$) and $\sum_{\mathbf{n} \in \NN_0^d} \mathbf{t}_{k-1}(\mathbf{n}) \, \mathbf{u}^{\mathbf{n}}$ (since $\maxnorm{\mathbf{u}} = \maxnorm{\boldsymbol{\rho}} < 1$, for all $\mathbf{u} \in \mathcal{C}(\boldsymbol{\rho})$); cf.~Proposition~\ref{prop:Sufficient_condition_for_unconditional_uniform_and_pointwise_convergence_of_multiple_power_series}~and~Proposition~\ref{prop:Unconditional_uniform_and_pointwise_convergence_of_vector_GFs}, respectively.

\subsection{Proof of Statement~\ref{statement:2.c}}

According to \eqref{eq:Vector_generating_function_F} and using Lemma~\ref{lem:Cauchy_integral_formulas}, we obtain  
\begin{align*} 
\mathbf{F}(\mathbf{z},w) & = \sum_{k \in \NN_0} {\mathbf{f}_k(\mathbf{z}) \, w^k} = \mathbf{f}_0(\mathbf{z}) + \sum_{k \in \NN} {\mathbf{f}_k(\mathbf{z}) \, w^k}   \\
& \stackrel{\mathclap{\eqref{eq:Functional_recurrence_vector_f_k}}}{=} \mathbf{f}_0(\mathbf{z}) + \sum_{\mathbf{r} \in {\mathcal{R}(\mathbf{m})}} \mathbf{z}^\mathbf{r} \sum_{(\boldsymbol{\ell},\nu) \in \mathcal{L}(\mathbf{r})}  \frac{\eta_{\mathbf{r},\boldsymbol{\ell},\nu}(\mathbf{z})}{\boldsymbol{\ell}!} \sum_{k \in \NN} {\mathbf{f}_{k-1}^{(\boldsymbol{\ell})}(\boldsymbol{\varphi}_{\mathbf{r},\boldsymbol{\ell},\nu}(\mathbf{z})) \, w^k}      \\
& = \mathbf{f}_0(\mathbf{z}) + w \sum_{\mathbf{r} \in {\mathcal{R}(\mathbf{m})}} \mathbf{z}^\mathbf{r} \sum_{(\boldsymbol{\ell},\nu) \in \mathcal{L}(\mathbf{r})}  \frac{\eta_{\mathbf{r},\boldsymbol{\ell},\nu}(\mathbf{z})}{\boldsymbol{\ell}!} \sum_{k \in \NN_0} {\mathbf{f}_{k}^{(\boldsymbol{\ell})}(\boldsymbol{\varphi}_{\mathbf{r},\boldsymbol{\ell},\nu}(\mathbf{z})) \, w^k}    \\
& \stackrel{\mathclap{\eqref{eq:Uniform_convergence_of_series_of_derivatives}}}{=}  \mathbf{f}_0(\mathbf{z}) + w \sum_{\mathbf{r} \in {\mathcal{R}(\mathbf{m})}} \mathbf{z}^\mathbf{r} \sum_{(\boldsymbol{\ell},\nu) \in \mathcal{L}(\mathbf{r})}  \frac{\eta_{\mathbf{r},\boldsymbol{\ell},\nu}(\mathbf{z})}{\boldsymbol{\ell}!} \mathbf{F}_{*}^{(\boldsymbol{\ell})}(\boldsymbol{\varphi}_{\mathbf{r},\boldsymbol{\ell},\nu}(\mathbf{z}),w)    ,
\end{align*} 
where in the last equality we also used Proposition~\ref{prop:Unconditional_uniform_convergence_implies_unconditional_pointwise_convergence}. 

This completes the proof of Theorem~\ref{thrm:Theorem_2}.


\vspace{2cc}

\noindent
{\bf Christos N. Efrem} \\
National Technical University of Athens, Greece \\
e-mail: \{the first two letters of the author's first name followed by the first three letters of the author's last name (all in lowercase)\}@mail.ntua.gr


\begin{thebibliography}{11}



\bibitem{Aizenberg-Yuzhakov_1983}
{\small {\sc L. A. A{\u\i}zenberg and A. P. Yuzhakov}, {\it Integral Representations and Residues in Multidimensional Complex Analysis}, American Mathematical Society, 1983.}

\bibitem{Berg-Meinardus_1994}
{\small {\sc L. Berg and G. Meinardus}, {\it Functional equations connected with the Collatz problem}, Results Math., {\bf 25} (1994), 1--12.}  

\bibitem{Berg-Meinardus_1995}
{\small {\sc L. Berg and G. Meinardus}, {\it The $3n+1$ Collatz problem and functional equations}, Rostock. Math. Kolloq., {\bf 48} (1995), 11--18.}  

\bibitem{Berg-Opfer_2013}
{\small {\sc L. Berg and G. Opfer}, {\it An analytic approach to the Collatz $3n+1$ problem for negative start values}, Comput. Methods Funct. Theory, {\bf 13} (2013), 225--236.} 

\bibitem{Collatz_1986}
{\small {\sc L. Collatz}, {\it {\"U}ber die entstehung des $(3n+1)$ problems}, J. Qufu Norm. Univ. Nat. Sci., {\bf 3} (1986), 9--11 (in Chinese).}

\bibitem{Egorychev_1984}
{\small {\sc G. P. Egorychev}, {\it Integral Representation and the Computation of Combinatorial Sums}, American Mathematical Society, 1984.}


\bibitem{Lagarias_1985}
{\small {\sc J. C. Lagarias}, {\it The $3x+1$ problem and its generalizations}, Amer. Math. Monthly, {\bf 92} (1985), no.~1, 3--23.}   

\bibitem{Lagarias_2010}
{\small {\sc J. C. Lagarias} (ed.), {\it The Ultimate Challenge: The $3x+1$ Problem}, American Mathematical Society, 2010.}

\bibitem{Lagarias_2011}
{\small {\sc J. C. Lagarias}, {\it The 3x+1 problem: An annotated bibliography (1963--1999)}, arXiv preprint, 2011, arXiv:math/0309224.}

\bibitem{Lagarias_2012}
{\small {\sc J. C. Lagarias}, {\it The 3x+1 problem: An annotated bibliography, II (2000--2009)}, arXiv preprint, 2012, arXiv:math/0608208.}

\bibitem{Lang_2005}
{\small {\sc S. Lang}, {\it Undergraduate Algebra}, Springer New York, 2005.}

\bibitem{Laurent-Thiebaut_2011}
{\small {\sc C. Laurent-Thi{\'e}baut}, {\it Holomorphic Function Theory in Several Variables: An Introduction}, Springer-Verlag, 2011.}


\bibitem{Range_1986}
{\small {\sc R. M. Range}, {\it Holomorphic Functions and Integral Representations in Several Complex Variables}, Springer New York, 1986.}

\bibitem{Riedel-Mahmoud_2023}
{\small {\sc M. Riedel and H. Mahmoud}, {\it Egorychev method: A hidden treasure}, La Matematica, {\bf 2} (2023), no.~4, 893--933.}

\bibitem{Rudin_1976}
{\small {\sc W. Rudin}, {\it Principles of Mathematical Analysis}, McGraw-Hill, Inc., 1976.}

\bibitem{Siegel_2019}
{\small {\sc M. C. Siegel}, {\it Conservation of singularities in functional equations associated to Collatz-type
dynamical systems; or, dreamcatchers for Hydra maps}, arXiv preprint, 2019, arXiv:1909.09733.}

\bibitem{Siegel_PhD_2022}
{\small {\sc M. C. Siegel}, {\it $\left(p,q\right)$-adic Analysis and the Collatz Conjecture}, Ph.D. thesis, University of Southern California, 2022, arXiv:2412.02902.} 

\bibitem{Siegel_2023}
{\small {\sc M. C. Siegel}, {\it The Collatz conjecture \& non-Archimedean spectral theory -~Part~I.5~- How to write the (weak) Collatz conjecture as a contour integral}, arXiv preprint, 2023, arXiv:2111.07883.}

\bibitem{Siegel_2024}
{\small {\sc M. C. Siegel}, {\it The Collatz conjecture \& non-Archimedean spectral theory -~Part~I~- Arithmetic dynamical systems and non-Archimedean value distribution theory}, p-Adic Numbers Ultrametric Anal. Appl., {\bf 16} (2024), no.~2, 143--199.}

\bibitem{Siegel_2025}
{\small {\sc M. C. Siegel}, {\it The Collatz conjecture \& non-Archimedean spectral theory -~Part~II~-
$(p,q)$-adic Fourier analysis and Wiener's Tauberian theorem}, p-Adic Numbers Ultrametric Anal. Appl., {\bf 17} (2025), no.~2, 187--232.}  


\bibitem{Tao_2022}
{\small {\sc T. Tao}, {\it Almost all orbits of the Collatz map attain almost bounded values}, Forum Math. Pi, {\bf 10} (2022), no.~12, 1--56.}

\bibitem{Titchmarsh_1939}
{\small {\sc E. C. Titchmarsh}, {\it The Theory of Functions}, Oxford University Press, 1939.}


\bibitem{Wirsching_2000}
{\small {\sc G. J. Wirsching}, {\it {\"U}ber das $3n + 1$ problem}, Elem. Math., {\bf 55} (2000), 142--155.} 


\end{thebibliography}
\end{document}